\documentclass[a4paper,11pt]{article}
\usepackage{amsmath,amssymb,amsthm,graphicx}
\usepackage[top=80pt, bottom=90pt, left=70pt, right=70pt]{geometry}
\usepackage{comment}
\usepackage{color}
\usepackage{multirow}
\usepackage[colorlinks=true,linkcolor=blue,citecolor=blue]{hyperref}
\usepackage{cleveref}
\usepackage{cases}
\usepackage{tikz}
\usepackage{url}
\usetikzlibrary{positioning}
\usetikzlibrary{calc}

\usepackage{enumitem}
\def\qed{\hfill $\Box$}
\def\qqed{\hfill $\blacksquare$}

\newcommand{\R}{\mathbf{R}}

\newcommand{\Z}{\mathbf{Z}}
\newcommand{\I}{\mathcal{I}}
\newcommand{\B}{\mathcal{B}}
\newcommand{\X}{\mathcal{X}}
\DeclareMathOperator{\dom}{dom}
\newcommand{\VIAP}{\textrm{VIAP}}
\newcommand{\argmin}{\operatornamewithlimits{argmin}}
\newcommand{\argmax}{\operatornamewithlimits{argmax}}

\newtheorem{theorem}{\bfseries Theorem} 
\newtheorem{lemma}[theorem]{\bfseries Lemma} 

\theoremstyle{definition}
\newtheorem{remark}[theorem]{\bfseries Remark} 

\crefname{theorem}{Theorem}{Theorems}
\crefname{proposition}{Proposition}{Propositions}
\crefname{lemma}{Lemma}{Lemmas}
\crefname{remark}{Remark}{Remarks}
\crefname{section}{Section}{Sections}
\crefname{figure}{Figure}{Figures}

\usepackage{bm}
\allowdisplaybreaks

\begin{document}
	\title{Optimal matroid bases with intersection constraints: \\Valuated matroids, M-convex functions,
and their applications\thanks{A preliminary version of this paper appears in the proceedings of the 16th Annual Conference on Theory and Applications of Models of Computation (TAMC 2020).
This work was done while Yuni Iwamasa was at National Institute of Informatics.}}
	\author{Yuni Iwamasa\thanks{Department of Communications and Computer Engineering, Graduate School of Informatics, Kyoto University, Kyoto 606-8501, Japan.
			Email: \texttt{iwamasa@i.kyoto-u.ac.jp}}
		\and Kenjiro Takazawa\thanks{Department of Industrial and Systems Engineering,
			Faculty of Science and Engineering,
			Hosei University, Tokyo
			184-8584, Japan.
			Email: \texttt{takazawa@hosei.ac.jp}}}
	\date{\today}
	\maketitle

\begin{abstract}
For two matroids $M_1$ and $M_2$ with the same ground set $V$ 
and 
two cost functions $w_1$ and $w_2$ on $2^V$, 
we consider 
the problem of finding bases $X_1$ of $M_1$ and $X_2$ of $M_2$ 
minimizing $w_1(X_1)+w_2(X_2)$
subject to a certain cardinality constraint on their intersection $X_1 \cap X_2$. 
For this problem,
Lendl, Peis, and Timmermans (2019) 
discussed modular cost functions:
they 
reduced the problem to weighted matroid intersection for the case 
where the cardinality constraint is $|X_1 \cap X_2|\le k$ or $|X_1 \cap X_2|\ge k$; 
and 
designed a new primal-dual algorithm 
for the case where the constraint is $|X_1 \cap X_2|=k$.

The aim of this paper is to generalize the problems to have nonlinear convex cost functions, 
and to comprehend them from the viewpoint of discrete convex analysis. 
We prove that each generalized problem can be solved via valuated independent assignment, valuated matroid intersection, 
or $\mathrm{M}$-convex submodular flow, 
to offer a comprehensive understanding of weighted matroid intersection with 
intersection constraints.
We also show the NP-hardness of some variants of these problems, 
which clarifies the coverage of discrete convex analysis for those problems. 
Finally, 
we present applications of our generalized problems in the recoverable robust matroid basis problem, combinatorial optimization problems with interaction costs, and matroid congestion games.
\end{abstract}
\begin{quote}
\end{quote}
\section{Introduction}
\label{SECintro}
{\it Weighted matroid intersection} is one of the most fundamental combinatorial optimization problems solvable in polynomial time.
This problem generalizes a number of tractable problems
including the maximum-weight bipartite matching and minimum-weight arborescence problems.
The comprehension of mathematical structures of weighted matroid intersection, e.g., Edmonds' intersection theorem~\cite{Edm70} and combinatorial primal-dual algorithm~\cite{Fra81,JORSJ/IT76,MP/L75}, 
contributes to the development of algorithmics in combinatorial optimization
as well as matroid theory.

In this article,
we represent a {\it matroid}
by a pair of its ground finite set $V$ and  {\it base family} $\B \subseteq 2^V$, 
which satisfy the following exchange axiom:
for $X, Y \in \B$ and $v \in X \setminus Y$,
there exists $u \in Y \setminus X$ such that
$X \setminus \{v\} \cup \{u\} \in \B$.
For a base family $\B$,
the set family $\I := \{ I \subseteq B \mid B \in \B \}$ is called the {\it independent set family}. 
It is well known (see e.g., \cite{book/Oxley11}) that 
the base family uniquely determines the corresponding independent set family, and vice versa.
Hence in this paper we also represent a matroid by the pair $(V, \I)$ of the ground set $V$ and the independent set family $\I$.

Let $M_1 = (V, \mathcal{B}_1)$ and $M_2 = (V, \mathcal{B}_2)$ be 
matroids
on $V$ with base families $\mathcal{B}_1$ and $\mathcal{B}_2$, respectively. 
Also let $w_1$ and $w_2$ be 
weight functions
on $V$
and 
$k$ a nonnegative integer. 
A weight function $w : V \rightarrow \R$ is also regarded as a {\it modular} function $w: 2^V \rightarrow \R$ 
defined by $w(X) = \sum_{v \in X} w(v)$ for each $X \subseteq V$. 
Recently, Lendl, Peis, and Timmermans~\cite{arxiv/LPT19}
have introduced the following variants of weighted matroid intersection,
in which a cardinality constraint is imposed on the intersection:
\begin{align*}
({\rm W}_{=k}) \qquad
&\begin{array}{lll}
\text{Minimize} \quad& w_1(X_1) + w_2(X_2)\\
\text{subject to} \quad& X_i \in \mathcal{B}_i \quad &(i = 1,2),\\
& |X_1 \cap X_2| = k.
\end{array} 
\\
({\rm W}_{\geq k}) \qquad
&\begin{array}{lll}
\text{Minimize} \quad& w_1(X_1) + w_2(X_2)\\
\text{subject to} \quad& X_i \in \mathcal{B}_i \quad &(i = 1,2),
\\
&|X_1 \cap X_2| \geq k.
\end{array}
\\
({\rm W}_{\leq k}) \qquad
&\begin{array}{lll}
\text{Minimize} \quad& w_1(X_1) + w_2(X_2)\\
\text{subject to} \quad& X_i \in \mathcal{B}_i \quad &(i = 1,2),\\
& |X_1 \cap X_2| \le k. 
\end{array}
\end{align*}
We remark here that the tractability of $({\rm W}_{=k})$
implies that of $({\rm W}_{\geq k})$ and $({\rm W}_{\leq k})$.
Indeed, for example,
we obtain an optimal solution for $({\rm W}_{\geq k})$ for $k = \ell$
by solving $({\rm W}_{=k})$
for
$k = \ell, \ell + 1, \dots, |V|$
and returning a minimum solution over them.

The motivation of these problems comes from 
the \emph{recoverable robust matroid basis problem}~\cite{book/Busing11}. 
Lendl et al.~\cite{arxiv/LPT19} showed that
$({\rm W}_{=k})$, $({\rm W}_{\geq k})$, and $({\rm W}_{\leq k})$
are strongly polynomial-time solvable:
they developed a new primal-dual algorithm
for $({\rm W}_{=k})$; 
and 
reduced 
$({\rm W}_{\geq k})$ and~$({\rm W}_{\leq k})$ to
weighted matroid intersection. 
By this result,
they affirmatively settled an open question on the strongly polynomial-time solvability of
the recoverable robust matroid basis problem under interval uncertainty representation~\cite{HKZ17JOCO,HKZ17OL}.

Lendl et al.~\cite{arxiv/LPT19} further discussed 
two kinds of generalizations of the above problems.
One is to consider more than two matroids. 
Let $n$ be a positive integer, 
and $[n] := \{1,2,\ldots,n\}$. 
For each $i \in [n]$, 
let 
$M_i = (V, \mathcal{B}_i)$ 
be a matroid with ground set $V$ and base family of $\B_i$. 
For instance,
$({\rm W}_{\leq k})$ can be generalized as follows. 
\begin{align*}
({\rm W}_{\leq k}^n) \qquad
&\begin{array}{lll}
\text{Minimize} \quad& \displaystyle\sum_{i =1}^n w_i(X_i) &\\
\text{subject to} \quad& X_i \in \mathcal{B}_i & (i \in [n]),\\
& \displaystyle\left|\bigcap_{i=1}^n X_i \right| \le k.&
\end{array}
\end{align*}
Generalizations of $({\rm W}_{=k})$ and $({\rm W}_{\geq k})$, which we name $({\rm W}_{=k}^n)$ and $({\rm W}_{\geq k}^n)$, respectively,
can be obtained in the same way. 
Lendl et al.~\cite{arxiv/LPT19} proved that 
$({\rm W}_{= k}^n)$ and $({\rm W}_{\geq k}^n)$ are NP-hard, 
whereas
$({\rm W}_{\leq k}^n)$ can be solved in strongly polynomial time. 
Indeed, 
they 
reduced $({\rm W}_{\leq k}^n)$ to
weighted matroid intersection.

The other is a polymatroidal generalization.
Let $B_1, B_2 \subseteq \Z^V$ be the base polytopes of some polymatroids on the ground set $V$. 
The following problem generalizes $({\rm W}_{\geq k})$, 
where 
$w_1$ and $w_2$ are linear functions on $\Z^V$.
\begin{align*}
({\rm P}_{\geq k}) \qquad
\begin{array}{lll}
\text{Minimize} \quad& w_1(x_1) + w_2(x_2)&\\
\text{subject to} \quad& x_i \in B_i \quad &(i = 1,2),\\
&\displaystyle\sum_{v \in V} \min \{ x_1(v), x_2(v) \} \geq k.&
\end{array}
\end{align*}
Again, 
generalizations of $({\rm W}_{=k})$ and $({\rm W}_{\leq k})$
can be obtained in the same manner,
and they 
are NP-hard.
Lendl et al.~\cite{arxiv/LPT19} proved that 
$({\rm P}_{\geq k})$ can be reduced to the \emph{polymatroidal flow problem}~\cite{Has82,LM82a,LM82b}, 
which is equivalent to the 
\emph{submodular flow problem} \cite{EG77} (see~\cite{book/Fujishige05}),
and thus can be solved in strongly polynomial time.

The aim of this paper is to offer a comprehensive understanding of the 
above problems 
in view of {\it discrete convex analysis} (\emph{DCA})~\cite{MP/M98,book/Murota03}, particularly focusing on {\it M-convexity}~\cite{AM/M96}.
DCA provides a theory of convex functions on the integer lattice $\Z^V$. 
M-convex functions 
play central roles in DCA and 
naturally appear in 
various research fields such as 
combinatorial optimization, economics, and game theory~\cite{incollection/M09,JMID/M16}.

M-convex functions are  
a quantitative generalization of matroids. 
The formal definition of M-convex functions is given as follows.
A function $f : \Z^V \rightarrow \R \cup \{ +\infty \}$ is said to be {\it M-convex}
if it satisfies the following generalization of the matroid exchange axiom:
for all $x = (x(v))_{v \in V}$ and $y = (y(v))_{v \in V}$ with $x, y \in \dom f$,
and all $v\in V$ with $x(v) > y(v)$,
there exists $u\in V$ with $x(u) < y(u)$
such that
\begin{align*}
    f(x) + f(y) \geq f(x - \chi_v + \chi_u) + f(y + \chi_v - \chi_u),
\end{align*}
where $\dom f$ denotes the effective domain $\{ x \in \Z^V \mid f(x) < +\infty \}$
of $f$
and $\chi_v$ the $v$-th unit vector for $v \in V$.
In particular,
if $\dom f$
is included in the hypercube $\{0,1\}^V$,
then $f$ is called a {\it valuated matroid}
\footnote{The original definition of a valuated matroid is an {\it M-concave function},
i.e., the negation of an M-convex function,
whose effective domain is included in the hypercube.}~\cite{AML/DW90,AM/DW92}.

In this paper,
we address 
the following $\mathrm{M}$-convex (and hence nonlinear) generalizations of $({\rm W}_{=k})$, $({\rm W}_{\geq k})$, $({\rm W}_{\leq k}^n)$, and $({\rm P}_{\geq k})$, 
and present their applications.
Let $\omega_1,\omega_2, \ldots, \omega_n$ 
be valuated matroids on $2^V$,
where we identify $2^V$ with $\{0,1\}^V$
by the natural correspondence between $X \subseteq V$ and $x \in \{0,1\}^V$; $x(v) = 1$ if and only if $v \in X$.
\begin{itemize}
    \item
    For $({\rm W}_{=k})$ and $({\rm W}_{\geq k})$, 
    by
    generalizing the weight functions $w_1$ and $w_2$
to valuated matroids $\omega_1$ and $\omega_2$, 
    we obtain: 
\begin{align*}
({\rm V}_{=k}) \qquad
\begin{array}{ll}
\text{Minimize} \quad& \omega_1(X_1) + \omega_2(X_2)\\
\text{subject to} \quad& |X_1 \cap X_2| = k;
\end{array}\\
({\rm V}_{\geq k}) \qquad
\begin{array}{ll}
\text{Minimize} \quad& \omega_1(X_1) + \omega_2(X_2)\\
\text{subject to} \quad& |X_1 \cap X_2| \geq k.
\end{array}
\end{align*}
Again observe that the tractability of~$({\rm V}_{=k})$
implies that of~$({\rm V}_{\geq k})$.

\item
For $({\rm W}_{\leq k}^n)$ (and hence~$({\rm W}_{\leq k})$ as well), 
in addition to generalizing $w_1, w_2, \dots, w_n$ to valuated matroids $\omega_1,\omega_2, \ldots, \omega_n$, 
we generalize the cardinality constraint $|\bigcap_{i = 1}^n X_i| \leq k$ 
to a matroid constraint. 
Namely, 
let $M = (V, \mathcal{I})$ be a new matroid,  
where $\mathcal{I}$ denotes its independent set family, 
and generalize~$({\rm W}_{\leq k}^n)$ as follows. 
\begin{align*}
({\rm V}_{\I}^n) \qquad
\begin{array}{ll}
\text{Minimize} \quad&\displaystyle\sum_{i = 1}^n \omega_i(X_i)\\
\text{subject to} \quad& \displaystyle\bigcap_{i=1}^n X_i \in \I.
\end{array}
\end{align*}

\item
It is also reasonable to take the intersection constraint 
into the objective function. 
Let $w\colon V \to \R$ be a 
weight function. 
The next problem is a variant of $({\rm V}_{\I}^n)$. 
\begin{align*}
({\rm V}^n(w)) \qquad
\begin{array}{ll}
\text{Minimize} \quad&\displaystyle\sum_{i = 1}^n \omega_i(X_i) + w\left( \bigcap_{i = 1}^n X_i \right).
\end{array}
\end{align*}

\item
Let $f_1$ and $f_2$ be M-convex functions on $\Z^V$
such that
$\dom f_1$ and $\dom f_2$ form the base polytopes of some polymatroids.
Also let $w \colon \Z^V \to \R$ be a 
linear function. 
Then, the following problem is a common generalization of $({\rm P}_{\geq k})$
and $({\rm V}_{\geq k})$. 
\begin{align*}
({\rm M}_{\geq k}(w)) \qquad
\begin{array}{ll}
\text{Minimize} \quad& f_1(x_1) + f_2(x_2) + w(\min \{x_1,x_2\})\\
\text{subject to} \quad&\displaystyle \sum_{v \in V} \min \{x_1(v), x_2(v) \} \geq k, 
\end{array}
\end{align*}
where $\min \{ x_1, x_2 \}\in \Z^V$ is a vector defined by 
$\min \{ x_1, x_2 \}=(\min \{x_1(v), x_2(v)\})_{v \in V} $.
\end{itemize}
The relations among the above problems
are given in Figure~\ref{fig:relation}.

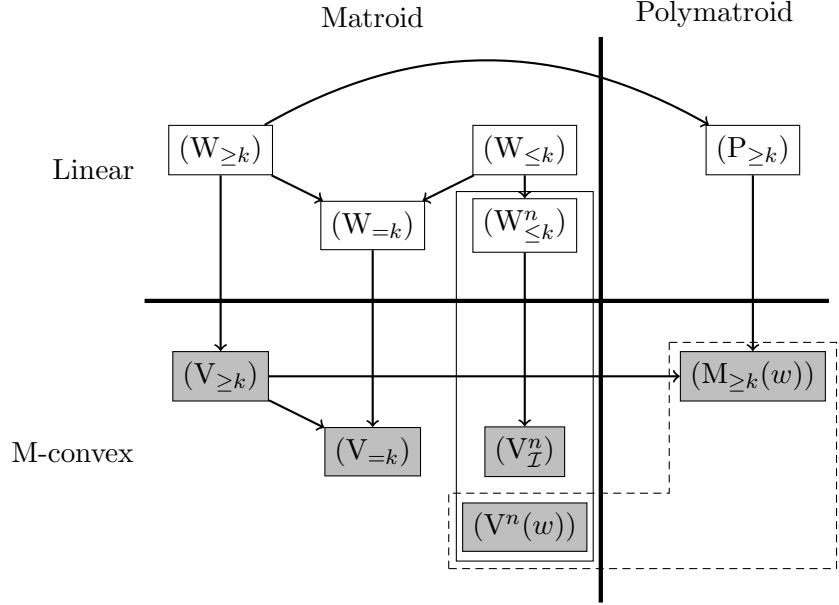
\begin{figure}
	\centering
	\begin{tikzpicture}[LPT/.style={draw, rectangle}, IT/.style={draw, fill=lightgray, rectangle}, edge/.style={thick, ->}]
	\coordinate (O) at (0,0);
	\coordinate (U) at ($ (O) + (0, 3.5) $);
	\coordinate (D) at ($ (O) - (0, 4) $);
	\coordinate (L) at ($ (O) - (6, 0) $);
	\coordinate (R) at ($ (O) + (3, 0) $);
	\coordinate (UL) at ($(U) + (L)$);
	\coordinate (DL) at ($(D) + (L)$);
	\coordinate (UR) at ($(U) + (R)$);
	\coordinate (DR) at ($(D) + (R)$);
	\coordinate [label=above:{Matroid}] (cube) at ($ (U)!0.5!(UL) $);
	\coordinate [label=above:{Polymatroid}] (poly) at ($ (U)!0.5!(UR) $);
	\coordinate [label=left:{Linear}] (linear) at ($ (UL)!0.5!(L) $);
	\coordinate [label=left:{M-convex}] (M-convex) at ($ (DL)!0.5!(L) $);
 	\draw[ultra thick] (U) -- (D);
 	\draw[ultra thick] (L) -- (R);
 	\node[LPT] (p1) at ($ (O) - (3, -1) $) {$({\rm W}_{=k})$};
 	\node[LPT] (p2) at ($ (p1) + (-2, 1) $) {$({\rm W}_{\geq k})$};
 	\node[LPT] (p3) at ($ (p1) + (2, 1) $) {$({\rm W}_{\leq k})$};
    \node[LPT] (p4) at ($ (O) + (2, 2) $) {$({\rm P}_{\geq k})$};
	\node[LPT] (p5) at ($ (p3) + (0, -1) $) {$({\rm W}_{\leq k}^n)$};
    \node[IT] (p6) at ($ (O) - (3, 2) $) {$({\rm V}_{=k})$};
    \node[IT] (p7) at ($ (p6) + (-2, 1) $) {$({\rm V}_{\geq k})$};
    \node[IT] (p8) at ($ (p6) + (2, 0) $) {$({\rm V}_{\I}^n)$};
	\node[IT] (p9) at ($ (p8) + (0, -1) $) {$({\rm V}^n(w))$};
 	\node[IT] (p10) at ($ (O) + (2, -1) $) {$({\rm M}_{\geq k}(w))$};
    
    \draw ($(p5) - (0.9, -0.45)$) rectangle ($(p9) + (0.9, -0.45)$);
    \draw[densely dashed] ($(p9) - (1, 0.55)$) -| ($(p10) + (1.1, 0.45)$) -- ($(p10) + (-1.1, 0.45)$);
    \draw[densely dashed] ($(p9) - (1, 0.55)$) -- ($(p9) - (1, -0.45)$) -| ($(p10) + (-1.1, 0.45)$);
	
	\draw[edge] (p2) -- (p1);
	\draw[edge] (p3) -- (p1);
	\draw[thick, ->, bend left] (p2) to (p4);
	\draw[edge] (p3) -- (p5);
	\draw[edge] (p1) -- (p6);
	\draw[edge] (p2) to (p7);
	\draw[edge] (p7) -- (p6);
	\draw[edge] (p5) -- (p8);
	\draw[edge] (p7) -- (p10);
	\draw[edge] (p4) -- (p10);
	\end{tikzpicture}
	\caption{This figure shows the relations among the problems discussed in this paper.
	The problems filled with gray are 
	new problems introduced in this paper.
	Each directed solid edge means that
	the problem at its head is a generalization of that at its tail.
	The terms ``Matroid'' and ``Polymatroid'' in the figure represent that
	the effective domain of the objective functions of the problem
	is essentially included in $\{0,1\}^V$ and in $\Z^V$,
	respectively.
	The terms ``Linear'' and ``M-convex'' represent that
	the functions used in the problem are linear (or modular) and M-convex (or valuated matroids), respectively.
	In $({\rm W}_{\leq k}^n)$, $({\rm V}_{\I}^n)$, and $({\rm V}^n(w))$,
	which are included in the solid rectangle,
	more than two modular functions or valuated matroids can appear in the 
	objective function.
	In $({\rm V}^n(w))$ and~$({\rm M}_{\geq k}(w))$,
	which are included in the dotted polygon,
	an additional modular/linear function $w$ appears in the 
	objective function.
	}
	\label{fig:relation}
\end{figure}

Our main contribution is to show the tractability of these generalized problems.
\begin{theorem}
\label{thm:P}
    There exist strongly polynomial-time algorithms to solve
   $({\rm V}_{=k})$, $({\rm V}_{\geq k})$, $({\rm V}_{\I}^n)$,
   and
    $({\rm V}^n(w))$ for $w \ge 0$, 
    and
    a weakly polynomial-time algorithm to
    solve $({\rm M}_{\geq k}(w))$ for $w \le 0$. 
\end{theorem}

The algorithms for $({\rm V}_{=k})$ and 
$({\rm V}_{\geq k})$ are based on 
\emph{valuated independent assignment}~\cite{SIDMA/M96_1,SIDMA/M96_2}, 
that for $({\rm V}_{\I}^n)$ and $({\rm V}^n(w))$ on \emph{valuated matroid intersection}~\cite{SIDMA/M96_1,SIDMA/M96_2}, 
and 
that for~$({\rm M}_{\geq k}(w))$ on \emph{M${}^\natural$-convex submodular flow}~\cite{Comb/M99}. 
We remark that 
valuated independent assignment generalizes valuated matroid intersection, 
and 
M${}^\natural$-convex submodular flow 
is a further generalization.
Besides this, 
the following facts are of theoretical interest. 
\begin{itemize}
    \item 
    If we apply our algorithm for $({\rm V}_{=k})$ to 
    the special case $({\rm W}_{=k})$, 
    we obtain a primal-dual algorithm which is essentially the same 
    as that in~\cite{arxiv/LPT19}, 
    but builds upon a slightly different optimality condition. 
    Details are described in Section~\ref{SEC:LPT}.
    
    \item
    For $({\rm V}^n(w))$ with $w \ge 0$, 
    our reduction results in a problem which is beyond 
    weighted matroid intersection 
    even if $\omega_i$ is a weight function on $V$ with a size constraint, 
    i.e., $\omega_i$ is described as 
    \begin{align}
        \label{eq:size}
    \omega_i(X) = 
    \begin{cases}
    \displaystyle 
    \sum_{v \in X} \omega_i(v) & \text{if $|X|=r_i$},\\
    +\infty & \text{if $|X| \neq r_i$},
    \end{cases}
    \end{align}
    with some nonnegative integer $r_i$
    for each $i \in [n]$. 
That is, 
this special case of 
$({\rm V}^n(w))$ is of interest in the sense that it 
does not require 
matroids to define, 
but requires valuated matroids to solve. 

\item
It might also be interesting that $({\rm V}_{\I}^n)$ can be solved in 
polynomial time when $n \ge 3$, 
in spite of the fact that matroid intersection for more than two matroids is NP-hard. 
\end{itemize}

We also demonstrate that 
the tractability of $({\rm V}^n(w))$ and $({\rm M}_{\geq k}(w))$ 
relies on the assumptions on $w$ 
($w \ge 0$ and $w \le 0$, 
respectively), 
by showing 
their NP-hardness for the general case.

\begin{theorem}
\label{thm:NP}
Problems $({\rm V}^n(w))$ and~$({\rm M}_{\geq k}(w))$ are NP-hard in general.
\end{theorem}

We then present applications of our generalized problems to
the recoverable robust matroid basis problem,
\emph{combinatorial optimization problems with interaction costs} (\emph{COPIC})~\cite{LCP19},
and
\emph{matroid congestion games}~\cite{ARV08}.  
First 
we provide a generalization of 
a certain class of the recoverable robust matroid basis problem in which the cost functions are
M-convex functions.
This is a special case of
$({\rm M}_{\geq k}(w))$,
and thus can be solved in polynomial time. 
We next
reduce 
a certain generalized case of the COPIC with \emph{diagonal costs} 
to~$({\rm V}^n(w))$ and~$({\rm M}_{\geq k}(w))$, 
to provide a generalized class of COPIC which can be solved in polynomial time.
Finally,
we show that 
computing the socially optimal state in a 
certain generalized model of matroid congestion games 
can be reduced to (a generalized version of) $({\rm V}^n(w))$ for $w \geq 0$, 
and 
thus can be done in polynomial time.

The rest of the paper is organized as follows.
Section~\ref{sec:prel} provides several fundamental facts on 
valuated matroids and
M-convex functions. 
In Section~\ref{sec:VIAP}, 
we present algorithms for solving 
$({\rm V}_{=k})$ and $({\rm V}_{\geq k})$ based on a valuated independent assignment algorithm. 
Sections~\ref{sec:VMI} and~\ref{sec:MCSF} are devoted to the reductions of 
$({\rm V}_{\I}^n)$ and $({\rm V}^n(w))$ 
for $w \ge 0$ to valuated matroid intersection,
and $({\rm M}_{\geq k}(w))$ for $w \le 0$ to M${}^\natural$-convex submodular flow,
respectively. 
We then prove that $({\rm V}^n(w))$ and $({\rm M}_{\geq k}(w))$ are in general NP-hard in Section~\ref{sec:NP}.
In Section~\ref{sec:app},
we present applications of our generalized problems in recoverable robust matroid basis problems, combinatorial optimization problems with interaction costs, and matroid congestion games.
Finally, in Section~\ref{sec:discussion},
we pose 
open problems which look similar to those discussed in this paper.

\section{Preliminaries}\label{sec:prel}
We prepare several facts and terminologies on valuated matroids and
$\mathrm{M}$-convex functions.
We have already used $\Z$ and $\R$ to denote the sets of integers and
real numbers, respectively.
The set of nonnegative integers are denoted by $\Z_+$,
and those of nonnegative real numbers and nonpositive real numbers by
$\R_+$, and $\R_-$,
respectively.
Recall the definition of M-convex functions described in Section \ref{SECintro}.
For an M-convex function $f$,
all members in $\dom f$
have the same ``cardinality,''
that is, there exists some integer $r$
such that $\sum_{v \in V} x(v) = r$ for all $x \in \dom f$.
We refer to $r$ as the {\it rank} of $f$. 
In $({\rm M}_{\geq k}(w))$,
we require that
$\dom f_1$
and $\dom f_2$
form the base polytopes of some polymatroids.
This condition is
equivalent to
$\dom f_1 \subseteq \Z_+^V$ and $\dom f_2 \subseteq \Z_+^V$.

Recall that a valuated matroid is an M-convex function 
defined on $2^V$. 
\emph{Valuated matroid intersection} \cite{SIDMA/M96_1,SIDMA/M96_2} 
is a generalization of weighted matroid intersection defined as follows: 
Given two valuated matroids $\omega_1$ and $\omega_2$ on $2^V$,
find $X \subseteq V$ minimizing the sum $\omega_1(X) + \omega_2(X)$. 

We next define the valuated independent assignment problem \cite{SIDMA/M96_1,SIDMA/M96_2}. 
Let $G = (V_1, V_2; E)$ be a bipartite graph,
$\omega_1 \colon 2^{V_1} \to \mathbf{R}\cup \{+\infty\}$ and $\omega_2 \colon 2^{V_2} \to \mathbf{R}\cup \{+\infty\}$ be valuated matroids, 
and $w \colon E \to \mathbf{R}$ be a weight function.
The {\it valuated independent assignment problem} parameterized by an integer $k$,
referred to as
$\VIAP(k)$,
is described as follows. 
\begin{align*}
\VIAP(k) \qquad
\begin{array}{ll}
\text{Minimize} \quad& \omega_1(X_1) + \omega_2(X_2) + w(F)\\
\text{subject to} \quad&\text{$F\subseteq E$ is a matching of $G$ with $\partial F \subseteq X_1 \cup X_2$},\\
&|F| = k,
\end{array}
\end{align*}
where $\partial F$ denote the set of the endpoints of $F \subseteq E$.
As mentioned in Section \ref{SECintro}, 
$\VIAP(k)$ is a generalization of valuated matroid intersection,
and both of them can be solved in strongly polynomial time~\cite{SIDMA/M96_1,SIDMA/M96_2}.

A function $f : \Z^V \rightarrow \R \cup \{ +\infty \}$ 
is said to be
{\it M${}^\natural$-convex}~\cite{MOR/MS99}
if it satisfies 
the following weaker exchange axiom:
for all $x = (x(v))_{v \in V}$ and $y = (y(v))_{v \in V}$ with 
$x,y \in \dom f$,
and all $v\in V$ with $x(v) > y(v)$,
it holds that
\begin{align*}
    f(x) + f(y) \geq f(x - \chi_v) + f(y + \chi_v),
\end{align*}
or
there exists $u \in V$ with $x(v) < y(v)$
such that
\begin{align*}
    f(x) + f(y) \geq f(x - \chi_v + \chi_u) + f(y + \chi_v - \chi_u).
\end{align*}
It is clear from the definition that M${}^\natural$-convexity slightly generalizes 
 M-convexity, 
while they are known to be essentially equivalent (see,  e.g.,~\cite{book/Murota03}, 
for details).
The following lemma shows one relation between M-convex and M${}^\natural$-convex functions. 
\begin{lemma}[\cite{MOR/MS99}]
\label{lem:MS}
For an M${}^\natural$-convex function $f$ and an integer $r$,
the restriction of $f$ to a hyperplane $\{ x \in \Z^V \mid \sum_{v \in V} x(v) = r \}$
is an M-convex function with rank $r$, 
if its effective domain is nonempty.
\end{lemma}

We close this section with the definition of  {\it M${}^\natural$-convex submodular flow}~\cite{Comb/M99}. 
Let $f$ be an M${}^\natural$-convex function on $\Z^V$
and
$G = (V, A)$ a directed graph endowed with an upper capacity function $\overline{c} : A \rightarrow \R \cup \{ +\infty \}$,
a lower capacity function $\underline{c} : A \rightarrow \R \cup \{ -\infty \}$,
and a weight function $w : A \rightarrow \R$. 
For a vector $\xi \in \R^A$, 
define its boundary $\partial \xi \in \R^V$ by
\begin{align*}
\partial \xi (v) := \sum\{ \xi(a) \mid a \in A,\ \text{$a$ enters $v$ in $G$} \} - \sum\{ \xi(a) \mid a \in A,\ \text{$a$ leaves $v$ in $G$} \}
\end{align*}
for $v \in V$.
The 
{\it M${}^\natural$-convex submodular flow problem}
for $(f, G)$ is the following problem with variable $\xi \in \R^A$:
\begin{align*}
    \begin{array}{ll}
	\text{Minimize} & \displaystyle\quad f(\partial \xi) + \displaystyle \sum_{a \in A} w(a)\xi(a)\\
	\text{subject to} & \quad \underline{c}(a) \leq \xi(a) \leq \overline{c}(a). 
	\end{array}
\end{align*}
The M${}^\natural$-convex submodular flow problem 
is a further generalization of $\VIAP(k)$, 
and 
can be solved in weakly polynomial time~\cite{MPA/IMM05,SIOPT/IS03}.

\section{Solving $({\rm V}_{=k})$ and $({\rm V}_{\geq k})$ via valuated independent assignment}\label{sec:VIAP}
This section provides strongly polynomial-time algorithms for solving $({\rm V}_{=k})$ and $({\rm V}_{\geq k})$. 
For their special cases $({\rm W}_{=k})$ and $({\rm W}_{\geq k})$,
Lendl et al.~\cite{arxiv/LPT19}
showed the polynomial-time solvability:
they
developed a new algorithm specific to $({\rm W}_{=k})$, 
and 
reduced $({\rm W}_{\geq k})$ to weighted matroid intersection.
In this paper, 
building upon the DCA perspective, 
we show that both of the generalized problems~$({\rm V}_{=k})$ and $({\rm V}_{\geq k})$ fall in the framework of valuated independent assignment. 

\subsection{Strongly polynomial-time algorithms}\label{subsec:strongly poly-time}
We first present an algorithm for $({\rm V}_{\geq k})$. 
Given an instance of $({\rm V}_{\geq k})$, 
construct an instance of $\VIAP(k)$ as follows. 
Set a bipartite graph $G$ by $(V_1, V_2; \{ \{v^1, v^2\} \mid v \in V \})$,
where $V_i$ is a copy of $V$
and $v^i \in V_i$ is a copy of $v \in V$ for $i = 1,2$. 
By abuse of notation,
for $i=1,2$,
a subset $X_i$ of $V$ is regarded as a subset of $V_i$ as well, and 
$\omega_i$ is regarded as a valuated matroid on $2^{V_i}$.
Set $w(e) := 0$ for every edge $e$.
We now obtain an instance of $\VIAP(k)$ defined by $G$, $\omega_1$, $\omega_2$, and $w$.

One can see that,
if
$(X_1, X_2)$ is feasible for $({\rm V}_{\geq k})$,
i.e., $|X_1 \cap X_2| \geq k$,
then there is a matching $F$ of $G$ with $\partial F \subseteq X_1 \cup X_2$ and $|F| = k$,
i.e., there exists a feasible solution $(X_1, X_2, F)$ for $\VIAP(k)$.
On the other hand,
if $(X_1, X_2, F)$ is a feasible solution for $\VIAP(k)$,
then $(X_1, X_2)$ is feasible for~$({\rm V}_{\geq k})$.
Moreover the objective value of a feasible solution $(X_1, X_2)$ for~$({\rm V}_{\geq k})$
is equal to
that of any corresponding feasible solution $(X_1, X_2, F)$
for $\VIAP(k)$ since $w(e)$ is identically zero.

Thus, $({\rm V}_{\geq k})$ 
reduces to $\VIAP(k)$, 
and hence
can be solved in strongly polynomial time 
in the following way based on the augmenting path algorithm for $\VIAP(k)$~\cite{SIDMA/M96_1,SIDMA/M96_2};
see also~\cite[Theorem~5.2.62]{book/Murota00}.
Here let $X_1$ and $X_2$ be the minimizers of $\omega_1$ and $\omega_2$,
respectively, 
which can be found in a greedy manner. 
\begin{description}
	\item[Step 1:]
	If $|X_1 \cap X_2| \geq k$,
	then output $(X_1,X_2)$ and stop.
	Otherwise, let $X_1^j := X_1$ and $X_2^j := X_2$,
	where $j := |X_1 \cap X_2| < k$.
	\item[Step 2:]
	Execute the augmenting path algorithm for $\VIAP(k)$.
	Then we obtain
	a sequence $\left( (X_1^j, X_2^j), (X_1^{j + 1}, X_2^{j + 1}), \dots, (X_1^\ell, X_2^\ell) \right)$ of solutions,
	where $\left| X_1^{j'} \cap X_2^{j'} \right| = j'$ for $j' = j, j+1, \dots, \ell$.
	If $\ell < k$,
	then output ``$({\rm V}_{\geq k})$ is infeasible.''
	If $\ell \geq k$,
	then output $(X_1^k, X_2^k)$.
\end{description}
We describe the full behavior of the above algorithm in Appendix~\ref{appendix:algorithm description}.

Our algorithm for $({\rm V}_{=k})$ directly follows from 
this 
algorithm for $({\rm V}_{\geq k})$. 
Again let $X_1$ and $X_2$ be the minimizers of $\omega_1$ and $\omega_2$,
respectively. 
\begin{description}
    \item[Case 1 ($|X_1 \cap X_2| \leq k$):]
	Execute the augmenting path algorithm for $({\rm V}_{\geq k})$.
	If the algorithm detects the infeasibility of $({\rm V}_{\geq k})$,
	then output ``$({\rm V}_{=k})$ is infeasible.''
	Otherwise we obtain an optimal solution $(X_1^*, X_2^*)$ with $|X_1^* \cap X_2^*| = k$ for $({\rm V}_{\geq k})$,
	and output it.
	
	\item[Case 2 ($|X_1 \cap X_2| > k$):]
	Let $r_1$ be the rank of $\omega_1$ and $\overline{\omega_2}(X) := \omega_2(V \setminus X)$ for $X \subseteq V$,
	which is the dual valuated matroid of $\omega_2$.
	Note that $X_1$ and $V \setminus X_2$ are minimizers of $\omega_1$ and $\overline{\omega_2}$,
	respectively,
	and $X_1 \cap (V \setminus X_2) < r_1 - k$.
	Then apply Case 1 to $({\rm V}_{= r_1 - k})$ for $\omega_1$ and $\overline{\omega_2}$. 
\end{description}

It is clear that the above methods solve
$({\rm V}_{\geq k})$ and $({\rm V}_{=k})$ in strongly polynomial time.
We can provide the explicit time complexities as follows. 
The proof is deferred to Appendix \ref{appendix:algorithm description}.

\begin{theorem}\label{thm:V=k}
Problems $({\rm V}_{\geq k})$ and $({\rm V}_{=k})$ can be solved in
$O(|V|rk\gamma + |V| k \log |V|)$ time and $O(|V|r^2\gamma + |V| r \log |V|)$ time,
respectively,
where $r$ is the maximum of the ranks of $\omega_1$ and $\omega_2$
and $\gamma$ is the time required for computing the function value.
\end{theorem}

\begin{remark}\label{rmk:= NP-hard}
If we are given at least three valuated matroids,
then $({\rm V}_{=k})$ and $({\rm V}_{\geq k})$ (even $({\rm W}_{=k})$ and $({\rm W}_{\geq k})$) will be NP-hard,
since they can formulate the matroid intersection problem for three matroids.
Indeed,
let $M_1 = (V, \mathcal{B}_1), M_2 = (V, \mathcal{B}_2), M_3 = (V, \mathcal{B}_3)$ be matroids with the same rank $r$.
Then 
it is clear that 
there exist $X_1, X_2, X_3 \subseteq V$
with $X_1 \in \mathcal{B}_1, X_2 \in \mathcal{B}_2, X_3 \in \mathcal{B}_3$ and
$|X_1 \cap X_2 \cap X_3| \geq r$ (or $|X_1 \cap X_2 \cap X_3| = r$)
if and only if there exists $X \subseteq V$ with $X \in \mathcal{B}_1 \cap \mathcal{B}_2 \cap \mathcal{B}_3$.
\qqed
\end{remark}

We close this subsection by 
exhibiting two more solutions to $({\rm V}_{\geq k})$ and $({\rm V}_{= k})$. 
The first one is
a different reduction of $({\rm V}_{\geq k})$ to 
valuated matroid intersection, 
which requires the argument described in Section~\ref{sec:VMI} below.
Indeed, $({\rm V}_{\geq k})$ for $\omega_1$ and $\omega_2$
is equivalent to $({\rm V}_{\leq r_1 - k})$ for $\omega_1$ and $\overline{\omega_2}$,
where $r_1$ is the rank of $\omega_1$ and $\overline{\omega_2}$ is the dual 
valuated matroid
of $\omega_2$.
One can see that $({\rm V}_{\leq r_1 - k})$ is a special case of $({\rm V}_{\I}^n)$ in which $n = 2$ and $\I$ is the independent set family of the uniform matroid with rank $r_1 - k$.
Thus,
by the argument in Section~\ref{sec:VMI},
$({\rm V}_{\leq r_1 - k})$ can be viewed as a special instance of 
valuated matroid intersection.

The second one is a solution to $({\rm V}_{= k})$.
Recently,
L\'aszl\'o V\'egh~\cite{PrivateV} has presented a simple  $O(|V|r^2 \gamma + |V| r \log |V|)$-time algorithm for $({\rm W}_{= k})$. 
By extending this algorithm,  we 
can
obtain another solution to $({\rm V}_{= k})$, 
which is described in Appendix \ref{appendix:Laszlo}.

\subsection{Relation to the algorithm by Lendl et al.}
\label{SEC:LPT}

In this subsection,
we illustrate the similarity and difference between 
the algorithm of Lendl et al.~\cite{arxiv/LPT19} for $({\rm W}_{=k})$ and 
our algorithm applied to 
$({\rm W}_{=k})$. 
For a detailed description of the algorithm of Lendl et al.~\cite{arxiv/LPT19},  the readers are referred to Appendix \ref{appendix:LPT}. 

As is shown in Appendices \ref{appendix:algorithm description} and \ref{appendix:LPT}, 
both algorithms maintain potential functions on $V_1$ and $V_2$. 
Main differences appear 
in the optimality criteria, 
and in 
the
updating procedures of a solution 
and
potentials. 

The algorithm of Lendl et al.~\cite{arxiv/LPT19} is based on the following sufficient condition of the optimality:
A feasible solution $(X_1, X_2)$ of~$({\rm W}_{=k})$ is optimal
if
there exist a nonnegative potential function $q_1 : V_1 \rightarrow \R_+$, a nonpositive potential function $q_2 : V_2 \rightarrow \R_-$, and a nonnegative value $\lambda \in \R_+$ satisfying that
\begin{itemize}
    \item
    $q_1(v^1) = q_2(v^2) + \lambda$ for each $v \in V$,
    \item
    $X_1$ and $X_2$ are minimizers of $w_1 - q_1$ and $w_2 + q_2$, respectively, and
    \item
    $q_1(v^1) = 0$ for $v^1 \in X_1 \setminus X_2$ and $q_2(v^2) = 0$ for $v^2 \in X_2 \setminus X_1$.
\end{itemize}
We refer to such $(q_1, q_2)$ as an {\it LPT optimality witness} of $(X_1, X_2)$.

Our algorithm is based on the
following is the optimality criteria for $\VIAP(k)$ in \cite[Theorem~5.1]{SIDMA/M96_1}  specialized to $({\rm V}_{\geq k})$.
For
a set function $\omega$ on $2^V$ and a weight function $w$ on $V$,
we define the functions 
$\omega + p$ and $\omega - p$ on $2^V$ by 
\begin{align*}
     &(\omega + p)(X) := \omega(X) + p(X), \\
     &(\omega - p)(X) := \omega(X) - p(X) 
\end{align*}
for each $X \subseteq V$. 
\begin{lemma}
\label{lem:witness}
A feasible solution $(X_1, X_2)$ for~$({\rm V}_{\geq k})$ is optimal if and only if
there are potential functions $p_1 : V_1 \rightarrow \R$ and $p_2 : V_2 \rightarrow \R$
satisfying 
the following three conditions:
\begin{itemize}
    \item
    $p_1(v^1) = p_2(v^2)$ for each $v \in V$;
    \item
    $X_1$ and $X_2$ are minimizers of $\omega_1 - p_1$ and $\omega_2 + p_2$, respectively; 
    and
    \item
    there is $F \subseteq X_1 \cap X_2$
    such that $|F| = k$,
    $X_1 \setminus F \subseteq \argmin p_1$,
    and
    $X_2 \setminus F \subseteq \argmax p_2 $.
\end{itemize}
\end{lemma}
For an optimal solution $(X_1, X_2)$ for~$({\rm V}_{\geq k})$, 
we refer to a triple $(p_1, p_2, F)$ satisfying the three conditions in Lemma~\ref{lem:witness} as an {\it optimality witness} of $(X_1, X_2)$.

We first see that 
the potential functions in the algorithm of Lendl et al.\ 
 and 
those in our algorithm
are essentially equivalent.  
During our algorithm for $({\rm V}_{\geq k})$,
one can see that
the potential functions $p_1$ and $p_2$
particularly satisfy the following:
\begin{itemize}
    \item
    $p_1(v^1) = p_2(v^2)$ for each $v \in V$.
    \item
    $X_1$ and $X_2$ are minimizers of $\omega_1 - p_1$ and $\omega_2 + p_2$, respectively.
    \item
    $\min p_1 = 0$,
    $X_1 \setminus X_2 \subseteq \argmin p_1$,
    and
    $X_2 \setminus X_1 \subseteq \argmax p_2$.
\end{itemize}
Hence, if $(q_1, q_2)$ is an LPT optimality witness,
then $(p_1, p_2)$ with $p_1 = p_2 = q_1$ forms an optimality witness in our sense.
Conversely, if $(p_1, p_2)$ is an optimality witness appearing in our algorithm,
then $(q_1, q_2)$ with $q_1 = p_1$ and $q_2 = p_2 - \lambda$ forms an LPT optimality witness for $\lambda := \max p_2$.

We next sketch the difference 
in the
updating procedures, 
whose details are discussed in 
Appendix \ref{appendix:LPT}. 
The update phases of a solution
and of 
potential functions in the algorithm of Lendl et al.\ 
are completely separated, 
while our algorithm 
simultaneously
updates a solution and 
potential functions. 
This
difference 
in the updating procedures
leads to the difference 
in
the running-times of the algorithms;
the algorithm of Lendl et al.\ 
runs in $O(|V|^2 r^2)$ time, 
and 
our algorithm in 
$O(|V|r^2 + |V| r \log |V|)$ time by 
Theorem~\ref{thm:V=k} with $\gamma = O(1)$,
which is better than that of Lendl et al.

\section{Reducing  $({\rm V}_{\I}^n)$ and~$({\rm V}^n(w))$ to valuated matroid intersection}\label{sec:VMI}
In this section,
we 
present reductions of $({\rm V}_{\I}^n)$ and $({\rm V}^n(w))$ for $w\ge 0$
to valuated matroid intersection, 
which implies 
strongly
polynomial-time algorithms for 
those problems.
We begin with the following result for valuated matroid intersection.
\begin{lemma}[\cite{SIDMA/M96_2}; see also Theorem~\ref{thm:V=k}]\label{lem:VMI}
Let $\omega$ and $\omega'$ be valuated matroids on $2^V$ with rank $r$
and $\gamma$ the time required for computing the function value.
Valuated matroid intersection for 
$\omega$ and $\omega'$
can be solved in $O(|V|r^2 \gamma + |V|r \log |V|)$ time.
\end{lemma}

For the reductions,
we need to prepare a pair of valuated matroids for each problem.
One valuated matroid is common in the reductions of $({\rm V}_{\I}^n)$ and $({\rm V}^n(w))$,
which is defined as follows.
Let $V_1,V_2,\ldots,V_n$ be $n$ disjoint copies of $V$, 
and 
let $\tilde{V} = \bigcup_{i\in [n]}V_i$. 
By abuse of notation, 
for $n$ subsets $X_1,X_2,\ldots,X_n \subseteq V$, 
we denote by 
$(X_1,X_2,\ldots, X_n)$
a subset of $\tilde{V}$ composed of the copies of 
$X_i$ included in $V_i$ ($i\in [n]$).
Let us define a valuated matroid $\tilde{\omega}$ by the disjoint sum of $\omega_1, \omega_2, \dots, \omega_n$.
That is, $\tilde{\omega}$ is a function on $2^{\tilde{V}}$ defined by
\begin{align*}
\tilde{\omega}(X_1, X_2, \dots, X_n) := \omega_1(X_1) + \omega_2(X_2) + \cdots + \omega_n(X_n)
\end{align*}
for each $(X_1, X_2, \dots, X_n) \subseteq \tilde{V}$.
It follows that $\tilde{\omega}$
is a valuated matroid with rank $r := \sum_{i = 1}^n r_i$,
where $r_i$ is the rank of $\omega_i$.

We then provide the other valuated matroid used in the reduction of $({\rm V}_{\I}^n)$.
Define a set system $\tilde{M} = (\tilde{V}, \tilde{\B})$ by
\begin{align*}
    \tilde{\B} =\left\{(X_1, X_2, \dots, X_n)\ \middle|\  \displaystyle\text{$X_i \subseteq V$ ($i \in [n]$), $\bigcap_{i=1}^n X_i \in \I$, $\sum_{i = 1}^n |X_i| = r$} \right\}.
\end{align*}
It is clear that
$({\rm V}_{\I}^n)$ amounts to 
minimizing the sum of $\tilde{\omega}$ and $\delta_{\tilde{\B}}$,
where $\delta_{\tilde{\B}}$ denotes the indicator function of $\tilde{\B}$,
namely,
$$\delta_{\tilde{\B}}(X_1, X_2, \dots, X_n) = 
\begin{cases}
0 & \mbox{if $(X_1, X_2, \dots, X_n) \in \tilde{\B}$}, \\
+\infty& \mbox{otherwise}.
\end{cases}$$
Thus, 
what remains to be proved is 
that $\delta_{\tilde{\mathcal{B}}}$ is a valuated matroid, 
and
is derived from the following lemma.
\begin{lemma}\label{lem:matroid}
The set system $\tilde{M}=(\tilde{V}, \tilde{\B})$ is a matroid with 
base family $\tilde{\B}$.
\end{lemma}
\begin{proof}
Let $\tilde{\I} \subseteq \tilde{V}$ be a subset family obtained from $\tilde{\B}$ by removing the cardinality condition,
that is,
\begin{align*}
    \tilde{\I} =\left\{(X_1, X_2, \dots, X_n)\ \middle|\  \displaystyle\text{$X_i \subseteq V$ ($i \in [n]$), $\bigcap_{i=1}^n X_i \in \I$} \right\}.
\end{align*}
It suffices to show that $\tilde{\I}$ is an independent set family 
of a matroid,
i.e.,
$\tilde{\I}$ satisfies the following axioms:
\begin{itemize}
    \item
    $(\emptyset, \emptyset, \dots, \emptyset) \in \tilde{\I}$.
    \item
    $(X_1, X_2, \dots, X_n) \subseteq (Y_1, Y_2, \dots, Y_n) \in \tilde{\I}$ implies $(X_1, X_2, \dots, X_n) \in \tilde{\I}$.
    \item
    For every $(X_1, X_2, \dots, X_n), (Y_1, Y_2, \dots, Y_n) \in \tilde{\I}$
    with $\sum_{i=1}^n |X_i| < \sum_{i=1}^n |Y_i|$,
    there exist $i^* \in [n]$ and $v^* \in Y_{i^*} \setminus X_{i^*}$ such that
    $(X_1, \dots, X_{i^*} \cup \{v^*\}, \dots, X_n) \in \tilde{\I}$.
\end{itemize}
The first and second are clear.
We prove the third.
If there exist $i^* \in [n]$ and $v^* \in Y_{i^*} \setminus X_{i^*}$ such that
\begin{align*}
    \left( \bigcap_{i \in [n]\setminus \{i^*\}} X_i\right) \cap (X_{i^*} \cup \{ v^* \}) = \bigcap_{i=1}^n X_i,
\end{align*}
then $(X_1, \dots, X_{i^*} \cup \{v^*\}, \dots, X_n) \in \tilde{\I}$ 
follows from $\bigcap_{i=1}^n X_i \in \I$.

Suppose that such $i^*$ and $v^*$ do not exist.
For 
each
$v \in V$, 
denote by $X^{(v)}$ (resp. $Y^{(v)}$) the set of indices $i \in [n]$ with $v \in X_i$ (resp. $v \in Y_i$).
Then,
for each $v \in V$,
we have
$X^{(v)} \supseteq Y^{(v)}$
or
$X^{(v)} = [n] \setminus \{i\}$ for some $i \in [n]$. 
This implies that $|Y^{(v)}| > |X^{(v)}|$ if and only if $Y^{(v)} = [n]$ and $X^{(v)} = [n] \setminus \{i\}$ for some $i$.
It then follows from $\sum_{i=1}^n |X_i| = \sum_{v \in V}|X^{(v)}|$, $\sum_{i=1}^n |Y_i| = \sum_{v \in V}|Y^{(v)}|$,
and $\sum_{i=1}^n |X_i| < \sum_{i=1}^n |Y_i|$ 
that 
\begin{align*}
\left|\left(\bigcap_{i=1}^n Y_i\right) \setminus \left(\bigcap_{i=1}^n X_i\right) \right|
&= \left|\{ v \in V \mid Y^{(v)} = [n], X^{(v)} \subsetneq [n] \}\right|\\
&= \sum_{v \in V,\ |Y^{(v)}| > |X^{(v)}|} \left(|Y^{(v)}| - |X^{(v)}|\right)\\
&= \sum_{v \in V} \left(|Y^{(v)}| - |X^{(v)}|\right) + \sum_{v \in V,\ |X^{(v)}| > |Y^{(v)}|} \left(|X^{(v)}| - |Y^{(v)}|\right)\\
&\geq \sum_{i = 1}^n \left( |Y_i| - |X_i| \right) + \left|\left(\bigcap_{i=1}^n X_i\right) \setminus \left(\bigcap_{i=1}^n Y_i\right)\right|\\
&> \left|\left(\bigcap_{i=1}^n X_i\right) \setminus \left(\bigcap_{i=1}^n Y_i\right)\right| \\
&\ge 0. 
\end{align*}
Since $\bigcap_{i=1}^n X_i$ and $\bigcap_{i=1}^n Y_i$ belong to $\I$,
there exists 
$v^* \in (\bigcap_{i=1}^n Y_i) \setminus (\bigcap_{i=1}^n X_i)$
such that $(\bigcap_{i=1}^n X_i) \cup \{v^*\} \in \I$.
Let $i^* \in [n]$ be an index such that $v^*\in Y_{i^*} \setminus X_{i^*}$.
We then obtain $(X_1, \dots, X_{i^*} \cup \{v^*\}, \dots, X_n) \in \I$, 
since $(\bigcap_{i \in [n] \setminus \{i^*\}} X_i) \cap (X_{i^*} \cup \{v^*\}) = (\bigcap_{i=1}^n X_i) \cup \{v^*\} \in \I$.
\qed
\end{proof}
It follows from Lemma~\ref{lem:matroid} 
that the function $\delta_{\tilde{\B}}$
is a valuated matroid, 
and we conclude that $({\rm V}_{\I}^n)$
can be reduced to valuated matroid intersection. 
Thus we obtain the following 
theorem from
Lemma~\ref{lem:VMI}:
\begin{theorem}\label{thm:VI}
Problem $({\rm V}_{\I}^n)$ can be solved in $O(|V|nr^2 \gamma + |V|nr \log (|V|n))$ time,
where $\gamma$ is the time required for computing the function value.
\end{theorem}

\begin{remark}
    If we replace the constraint $\bigcap_{i=1}^n X_i \in \I$ in~$({\rm V}_{\I}^n)$ by $\bigcap_{i=1}^n X_i \in \mathcal{B}$,
where $\mathcal{B}$ is the base family of some matroid,
then the problem will be NP-hard even if $n = 2$,
since it can formulate the matroid intersection problem for three matroids. 
In other words, 
if we replace 
the intersection constraint $|X_1 \cap X_2| =k$ in $({\rm W}_{=k})$ 
with $X_1 \cap X_2 \in \B$, 
then the problem becomes NP-hard.
Also, recall Remark~\ref{rmk:= NP-hard}, 
which implies that the problem is NP-hard if 
$n \geq 3$ and the constraint is $|\bigcap_{i=1}^n X_i| = k$.
\qqed
\end{remark}

We next provide another valuated matroid used in the reduction of $({\rm V}^n(w))$.
A set family $\mathcal{L} \subseteq 2^V$ is said to be {\it laminar}
if $X \subseteq Y$, $X \supseteq Y$, or $X \cap Y = \emptyset$ holds for all $X, Y \in \mathcal{L}$.
A function $f : \Z^V \rightarrow \R \cup \{ +\infty \}$
is said to be {\it laminar convex} \cite[Section 6.3]{book/Murota03}
if
$f$ is representable as
\begin{align*}
f(x) = \sum_{X \in \mathcal{L}} g_X\left(\sum_{v \in X} x(v)\right) 
\quad 
\left(x \in \Z^V \right), 
\end{align*}
where $\mathcal{L} \subseteq 2^V$ is a laminar family on $V$,
and for each $X \in \mathcal{L}$, $g_X : \Z \rightarrow \R \cup \{ +\infty \}$ is a univariate discrete convex function, 
i.e., $g_X(k+1) + g_X(k-1) \geq 2 g_X(k)$ for every $k \in \Z$. 
A laminar convex function 
is a typical example of an M${}^\natural$-convex function 
and plays a key role here.

Define a function $\tilde{w}$ on $2^{\tilde{V}}$
by
\begin{align*}
\tilde{w}(X_1, X_2, \dots, X_n) := w\left(\bigcap_{i = 1}^n X_i\right)
\end{align*}
for each $(X_1, X_2, \dots, X_n) \subseteq \tilde{V}$.
It is clear that $({\rm V}^n(w))$
is equivalent to minimizing the sum of $\tilde{\omega}$ and the restriction of
$\tilde{w}$ to $\{ (X_1, X_2, \dots, X_n)  \mid \sum_{i=1}^n |X_i| = r \}$.
For the function $\tilde{w}$, the following holds.
\begin{lemma}\label{lem:laminar}
The function $\tilde{w}$ on $2^{\tilde{V}}$ is laminar convex if $w \geq 0$.
\end{lemma}
\begin{proof}
For each $v \in V$,
define a unary function $g_v : \Z \rightarrow \R$
by
\begin{align*}
g_v(x) :=
\begin{cases}
w(v) & \text{if $x = n$},\\
0 & \text{if $0 \leq x < n$},\\
+\infty & \text{otherwise}.
\end{cases}
\end{align*}
It follows from  $w(v) \geq 0$ that 
$g_v$ is discrete convex.
Moreover, 
one can see that
\begin{align}\label{eq:g}
\tilde{w}( X_1, X_2, \dots, X_n ) 
{}&{}= \sum_{v \in V} g_v \left( \left| \{ i \in [n] \mid v \in X_i \} \right| \right) 
\quad (\mbox{$X_i \subseteq V$, $i \in [n]$})
\end{align}
holds. 
Here we can regard
the right-hand side of~\eqref{eq:g} as the sum 
taken for $(\{v\},\{v\},\dots, \{v\}) \subseteq  \tilde{V}$ 
for every $v \in V$. 
Since the family $\{ (\{v\},\{v\},\dots, \{v\}) \mid v \in V \}$ is laminar on $\tilde{V}$,
we conclude that the right-hand side of~\eqref{eq:g} is a laminar convex function,
as required.
\qed
\end{proof}

By Lemmas~\ref{lem:MS} and~\ref{lem:laminar},
the restriction of
$\tilde{w}$ to $\{ (X_1, X_2, \dots, X_n) \mid \sum_{i=1}^n |X_i| = r \}$
is a valuated matroid on $2^{\tilde{V}}$
if $w \geq 0$.
Thus $({\rm V}^n(w))$ can be formulated as valuated matroid intersection problem for $\tilde{\omega}$ and $\tilde{w}$, 
establishing the tractability of $({\rm V}^n(w))$ in case of $w \geq 0$.
That is, 
we obtain the following theorem from Lemma~\ref{lem:VMI}.
\begin{theorem}
Problem $({\rm V}^n(w))$ with $w \geq 0$ can be solved in $O(|V|nr^2 \gamma + |V|nr \log (|V|n))$ time,
where $\gamma$ is the time required for computing the function value.
\end{theorem}

\begin{remark}
As mentioned in Section~\ref{SECintro}, 
even if $\omega_i$ is 
described as~\eqref{eq:size} 
for each $i \in [n]$, 
our reduction 
goes beyond 
the weighted matroid intersection framework. 
This is 
because the valuated matroid $\tilde{w}$ is not modular regardless whether 
$\omega_i$ is of the form~\eqref{eq:size}. 
That is, 
the concept of 
 M-convexity is essential  
 for us to establish 
 the tractability of 
 $({\rm V}^n(w))$, 
even when 
$\omega_i$ is of the form~\eqref{eq:size}. 
\qqed
\end{remark}

\section{Reducing  $({\rm M}_{\geq k}(w))$ to M${}^\natural$-convex submodular flow}\label{sec:MCSF}
In this section,
we 
prove that $({\rm M}_{\geq k}(w))$  
with $w \le 0$
can be solved in polynomial time 
by reducing it to M${}^\natural$-convex submodular flow.

Given an instance of $({\rm M}_{\geq k}(w))$ 
with $w \le 0$, 
construct an instance of the M${}^\natural$-convex submodular flow problem as follows. 
Let $V_1 := \{ v^1 \mid v \in V \}$ and $V_2 := \{ v^2 \mid v \in V \}$ be disjoint two copys of $V$.
We regard $f_1$ and $f_2$ as functions on $\Z^{V_1}$ and on $\Z^{V_2}$,
respectively.
Recall that $\dom f_1$ and $\dom f_2$ form the base polytopes of some polymatroids,
i.e.,
$\dom f_1 \subseteq \Z_+^V$ and $\dom f_2 \subseteq \Z_+^V$.
Let $r_i$ be the rank of $f_i$ for $i = 1,2$.
We define univariate functions $g_1$ and $g_2$ on $\Z$
by
\begin{align*}
g_1(p) :=
\begin{cases}
0 & \text{if $0 \leq p \leq r_2 - k$},\\
+\infty & \text{otherwise},
\end{cases}
\qquad
g_2(q) :=
\begin{cases}
0 & \text{if $0 \leq q \leq r_1 - k$},\\
+\infty & \text{otherwise}.
\end{cases}
\end{align*}
Let $s$ and $t$ be distinct elements not belonging to $V_1 \cup V_2$, 
and define a function $h$ on $\Z^{V_1 \cup \{s\} \cup V_2 \cup \{t\} }$
by the disjoint sum of $f_2, g_2$ with the simultaneous coordinate inversion and $f_1, g_1$,
i.e.,
\begin{align*}
h(x_1,p,x_2,q) := \left(f_1(-x_1) + g_1(-p)\right) + \left(f_2(x_2) + g_2(q)\right) 
\end{align*}
for each $x_1 \in \Z^{V_1}$, $x_2 \in \Z^{V_2}$, and $(p,q) \in \Z^{\{s,t\}}$.
It then follows that $h$ is an M${}^\natural$-convex function.
Indeed,
$g_1$ and $g_2$
are clearly M${}^\natural$-convex,
the simultaneous coordinate inversion 
keeps the M${}^\natural$-convexity (see, e.g., \cite[Theorem 6.13 (2)]{book/Murota03}),
and
the disjoint sum of two M${}^\natural$-convex functions is also M${}^\natural$-convex.
We then construct a directed bipartite graph $G = ( V_1 \cup \{s\}, V_2 \cup \{t\}; A)$
endowed with a weight function $\hat{w}\colon A \to \R$ defined
by 
\begin{align*}
    &{}A := \{ (v^1, v^2) \mid v \in V \} 
    \cup \{(v^1, t) \mid v \in V \}
    \cup \{ (s, v^2) \mid v \in V \}, \\ 
&{}\hat{w}(a) :=
\begin{cases}
w(v) & \text{if $a = (v^1, v^2)$},\\
0 & \text{otherwise}
\end{cases}
\qquad (a \in A).
\end{align*}
We now obtain  the following instance of the M${}^\natural$-convex submodular flow problem:
\begin{align}\label{eq:MCSF}
\begin{array}{lll}
\text{Minimize} &\displaystyle\quad
h(\partial \xi) + \sum_{a \in A}\hat{w}(a) \xi(a)\\
\text{subject to} &\quad \xi(a) \geq 0 \quad &(a \in A).
\end{array}
\end{align}

The following lemma shows that 
$({\rm M}_{\geq k}(w))$ with $w \leq 0$ is reduced to the problem~\eqref{eq:MCSF}, 
and thus 
establishes its tractability.
\begin{lemma}\label{lem:MCSF}
Problem $({\rm M}_{\geq k}(w))$ with $w \leq 0$
is equivalent to the problem~\eqref{eq:MCSF}. 
\end{lemma}
\begin{proof}
Given a
feasible solution $(x_1, x_2)$ of $({\rm M}_{\geq k}(w))$, 
construct a feasible solution $\xi$ of~\eqref{eq:MCSF}
by 
\begin{align}\label{eq:xi}
\xi(a) :=
\begin{cases}
\min \{ x_1(v), x_2(v) \} & \text{if $a = (v^1, v^2)$},\\
\max\{ 0, x_1(v) - x_2(v) \} & \text{if $a = (v^1, t)$}, \\
\max\{ 0, x_2(v) - x_1(v) \} & \text{if $a = (s, v^2)$}. 
\end{cases}
\end{align}
Then 
it is not difficult to see that 
$(x_1, x_2)$ in~$({\rm M}_{\geq k}(w))$ 
and $\xi$ in~\eqref{eq:MCSF} have the same 
objective values.

Conversely, take any feasible solution $\xi$ for~\eqref{eq:MCSF}.
If 
$\xi(v^1, t) = 0$ 
or 
$\xi(s, v^2) = 0$ 
holds for every $v \in V$,
then we can straightforwardly construct a feasible solution $(x_1, x_2)$ satisfying~\eqref{eq:xi}.
In this case,
the objective value of $(x_1, x_2)$ in $({\rm M}_{\geq k}(w))$ and that of $\xi$ in \eqref{eq:MCSF} are the same.
Suppose that there is $v \in V$ with $\xi(v^1, t) > 0$ and $\xi(s, v^2) > 0$.
Then define
a new feasible solution $\xi'$ of~\eqref{eq:MCSF} by
\begin{align*}
\xi'(a) :=
\begin{cases}
\xi(a) + \min \{ \xi(v^1, t), \xi(s, v^2) \} & \text{if $a = (v^1, v^2)$},\\
\xi(a) - \min \{ \xi(v^1, t), \xi(s, v^2) \} & \text{if $a = (v^1, t)$ or $a = (s, v^2)$}.
\end{cases}
\end{align*}
Clearly $\partial \xi = \partial \xi'$ holds.
It
follows from 
$w \leq 0$ that 
the objective value for $\xi'$ is at most that for $\xi$. 
It also follows that $\xi'(v^1, t) = 0$ or $\xi'(s, v^2) = 0$ holds for each $v \in V$.
We can thus construct a feasible solution $(x_1, x_2)$ satisfying~\eqref{eq:xi},
attaining the desired objective value.
Therefore we conclude
that $({\rm M}_{\geq k}(w))$ with $w \leq 0$ and the problem~\eqref{eq:MCSF} are equivalent.
\qed
\end{proof}

\section{The NP-hardness of~$({\rm V}^n(w))$ and~$({\rm M}_{\geq k}(w))$}\label{sec:NP}
In Sections~\ref{sec:VMI} and~\ref{sec:MCSF},
we have shown the tractability of $({\rm V}^n(w))$ with $w \ge 0$,
and that of $({\rm M}_{\geq k}(w))$ with $w \le 0$.
This section is devoted to showing that they are NP-hard in general.

\paragraph{The NP-hardness of~$({\rm V}^n(w))$.}
We prove that $({\rm V}^n(w))$ can formulate 
the problem of finding a maximum common independent set for three matroids, 
which is NP-hard.
Given three matroids $M_1, M_2, M_3$ on the same ground set $V$
with the base families $\B_1, \B_2, \B_3$,
respectively, 
construct an instance of $({\rm V}^n(w))$ as follows. 
Let 
$\delta_{\B_i}$ 
be the indicator function of $\B_i$
for $i = 1,2,3$, 
and 
define a weight function $w : V \rightarrow \R$ by $w(v) := -1$ for each $v \in V$. 
Clearly, 
$\delta_{\B_i}$ is a valuated matroid for each $i=1,2,3$. Thus, 
$\delta_{\B_i}$ ($i=1,2,3$) and $w$ define an instance of $({\rm V}^n(w))$ with $n=3$. 
It is straightforward to see that 
this instance of $({\rm V}^n(w))$ 
is equivalent to
the problem of finding a maximum comment independent set of $M_1, M_2, M_3$,
as required.

\paragraph{The NP-hardness of~$({\rm M}_{\geq k}(w))$.}
Lendl et al.\ \cite{arxiv/LPT19} proved that the following problem $({\rm P}_{= 0})$ is NP-hard: 
\begin{align*}
({\rm P}_{= 0}) \qquad
\begin{array}{lll}
\text{Minimize} \quad& w_1(x_1) + w_2(x_2)&\\
\text{subject to} \quad& x_i \in B_i \quad &(i = 1,2),\\
&\displaystyle\sum_{v \in V} \min \{ x_1(v), x_2(v) \} = 0,&
\end{array}
\end{align*}
where $B_1, B_2 \subseteq \Z^V$ are the base polytopes of some polymatroids on the ground set $V$,
and
$w_1$ and $w_2$ are linear functions on $\Z^V$. 
Here we prove that $({\rm P}_{= 0})$ can be reduced to $({\rm M}_{\geq k}(w))$. 

Given an instance of $({\rm P}_{= 0})$, 
construct an instance of $({\rm M}_{\geq k}(w))$ in the following way. 
For $i = 1,2$,
define a function $f_i \colon \Z^V \to \R \cup\{+\infty\}$ by 
\begin{align*}
f_i(x_i) := 
\begin{cases}
w_i(x_i) & \text{if $x_i \in B_i$}, \\
+\infty & \text{otherwise}.
\end{cases}
\end{align*}
It is not difficult to see that $f_1$ and $f_2$ are M-convex functions.
Set $k = 0$ and $w$ sufficiently large, e.g., $w(v) > \max_{x_1 \in B_1} f_1(x_1) + \max_{x_2 \in B_2} f_2(x_2)$ for each $v \in V$.
Denote 
an instance of $({\rm M}_{\geq k}(w))$ defined by these $f_1, f_2, k, w$  by (I).

Then, 
an optimal solution for the given instance of $({\rm P}_{= 0})$ can be obtained from that of (I).
If (I) is infeasible,
then clearly the instance of $({\rm P}_{= 0})$ is infeasible. 
Suppose that (I) has an optimal solution  
$(x_1^*,x_2^*)$ with bounded objective value. 
If 
$\sum_{v \in V} \min\{x_1^*(v), x_2^*(v)\} > 0$, 
it follows from the construction of $w$ that 
there is no $(x_1,x_2)$ 
satisfying
that 
$x_1\in B_1$, $x_2 \in B_2$, 
and $\sum_{v \in V} \min\{x_1^*(v), x_2^*(v)\} = 0$, 
which implies
that the given instance of $({\rm P}_{= 0})$ is infeasible. 
If 
$\sum_{v \in V} \min\{x_1^*(v), x_2^*(v)\} = 0$, 
then 
it is straightforward to see that 
$(x_1^*,x_2^*)$ is also an optimal solution for the instance of $({\rm M}_{\geq k}(w))$.

\section{Applications}\label{sec:app}

In this section, 
we present applications of our generalized problems 
in the recoverable robust matroid basis problem, combinatorial optimization problems with interaction costs, 
and matroid congestion games.

\subsection{Recoverable robust matroid basis problem}\label{subsec:recrob}
In the {\it recoverable robust matroid basis problem}~\cite{book/Busing11}, 
we are given a matroid $(V, \mathcal{B})$ with ground set $V$ and base family $\mathcal{B}\subseteq 2^V$, 
a weight function $w_1$ on $V$,
a family $\mathcal{W}$ of weight functions on $V$,
and a nonnegative integer $k$.
The recoverable robust matroid basis problem
is described as 
the following minimization problem with variable $X_1 \in \mathcal{B}$:
\begin{align}\label{eq:recrobmat}
\begin{array}{ll}
    \text{Minimize}\quad &\displaystyle w_1(X_1) + \max_{w_2 \in \mathcal{W}} \left\{\min_{X_2 \in \mathcal{B}, |X_1 \cap X_2| \geq k} w_2(X_2)\right\} \\
    \text{subject to}\quad & 
    {X_1 \in \mathcal{B}}. 
\end{array}
\end{align}

This problem simulates the following situation.
The family $\mathcal{W}$ represents the uncertainty of cost functions.
The actual cost function $w_2 \in \mathcal{W}$ is revealed after choosing a basis $X_1 \in \mathcal{B}$,
which costs $w_1(X_1)$.
In the recovery phase,
we rechoose a basis $X_2 \in \mathcal{B}$ that is not much different from the first basis $X_1$,
i.e.,\  $|X_1 \cap X_2| \geq k$,
which requires the additional cost $w_2(X_2)$.
The objective is to minimize the worst-case total cost $w_1(X_1) + w_2(X_2)$.

It is known~\cite{KKZ14}
that the recoverable robust matroid basis problem
is NP-hard even when $|\mathcal{W}|$ is constant and $\mathcal{B}$
is a base family of a graphic matroid.
Lendl et al.~\cite{arxiv/LPT19} observed that
the recoverable robust matroid basis problem
can be reduced to $({\rm W}_{\geq k})$
if the uncertainty set $\mathcal{W}$ has the interval uncertainty representation:
\begin{align*}
    \mathcal{W} = \{w \colon V \to \R \mid \underline{w}(v) \leq w(v) \leq \overline{w}(v) \text{ for each $v \in V$} \}.
\end{align*}
Indeed, 
in this case, 
\eqref{eq:recrobmat} can be described in the form of~$({\rm W}_{\geq k})$:
\begin{align*}
\begin{array}{ll}
    \text{Minimize}\quad & w_1(X_1) + \overline{w}(X_2) \\
    \text{subject to} \quad& X_1, X_2 \in \mathcal{B},\\
    &|X_1 \cap X_2| \geq k.
\end{array}
\end{align*}

Our result naturally gives its nonlinear and polymatroidal generalization.
Let $f_1$ and $f$ be M-convex functions on $\Z^V$ with $\dom f_1, \dom f \subseteq \Z_+^V$. 
Define a family $\mathcal{W}$ of M-convex functions by
\begin{align}
\label{eq:W M}
    \mathcal{W}=\{ f + w \mid w\colon \text{linear function on $\Z^V$ with }\underline{w} \leq w \leq \overline{w}\}, 
\end{align}
where $\underline{w}$ and $\overline{w}$ are linear functions on $\Z^{V}$.
Now, consider the following problem:
\begin{align}\label{eq:recrobval M}
\text{Minimize}\quad  f_1(x_1) + \max_{f_2 \in \mathcal{W}} \left\{\min \left\{ f_2(x_2) \ \middle|\ \sum_{v \in V} \min\{ x_1(v), x_2(v) \} \geq k  \right\}\right\}.
\end{align}
It then follows that
\eqref{eq:recrobval M} amounts to
\begin{align*}
\begin{array}{ll}
    \text{Minimize} \quad & f_1(x_1) + (f_2 + \overline{w})(x_2) \\
    \text{subject to}\quad &\displaystyle \sum_{v \in V} \min\{ x_1(v), x_2(v) \} \geq k.
\end{array}
\end{align*}
Since $f_2 + \overline{w}$ is M-convex,
this is a special case of
$({\rm M}_{\geq k}(w))$
with $w \leq 0$, 
and thus 
can be solved in weakly polynomial time.
In particular, if the objective function is defined on $2^V$, or equivalently,
$f_1$ and $f_2$ are valuated matroids,
then \eqref{eq:recrobval M}
is equivalent to $({\rm V}_{\geq k})$,
and can be solved in strongly polynomial time.

\begin{theorem}
The problem \eqref{eq:recrobval M} can be solved in weakly polynomial time when the uncertainty set $\mathcal{W}$ is in the form of \eqref{eq:W M}. 
In addition, if the objective function is defined on $2^V$,
then it can be solved in strongly polynomial time.
\end{theorem}

Let us also mention 
the following variant
of the recoverable robust matroid basis problem discussed in Lendl et al.~\cite{arxiv/LPT19}. 
Here, $w_1$ and $w_2$ are weight functions on $V$,
$\B_1$ and $\B_2$ are the base families of matroids on $V$,
and $c : \Z_+ \to \R \cup \{+\infty\}$ is a univariate function which is not necessarily linear. 
\begin{align*}
({\rm W}_{c}) \qquad
 &\begin{array}{lll}
 \text{Minimize} \quad& w_1(X_1) + w_2(X_2) + c(|X_1 \cap X_2|)\\
 \text{subject to} \quad& X_i \in \mathcal{B}_i \quad &(i = 1,2).
 \end{array} 
 \end{align*}
They showed that $({\rm W}_{c})$ can be solved in strongly polynomial time
by reducing it to $({\rm W}_{=k})$.
Indeed,
for each
$k = 0, 1, \dots, |V|$,
let $(X_1^k, X_2^k)$ be an optimal solution of $({\rm W}_{=k})$.
Then we have that $\min_{k = 0,1,\dots, |V|} \{w_1(X_1^k) + w_2(X_2^k) + c(k)\}$
is an optimal solution of $({\rm W}_{c})$.
We remark that the original form in \cite{arxiv/LPT19} includes
a penalty $C(|X_1 \triangle X_2|)$ in place of $c(|X_1 \cap X_2|)$,
where $C$ is also a univariate function.
It follows that $c$ and $C$ have a one-to-one correspondence defined
by a certain transformation.

Here we generalize
the weight functions $w_1$ and $w_2$ in $({\rm W}_{c})$
to valuated matroids $\omega_1$ and $\omega_2$ on $2^V$:
\begin{align*}
({\rm V}_{c}) \qquad
&\begin{array}{lll}
\text{Minimize} \quad& \omega_1(X_1) + \omega_2(X_2) + c(|X_1 \cap X_2|).
\end{array} 
\end{align*}
By the same argument for $({\rm W}_{c})$,
we can solve $({\rm V}_{c})$ in strongly polynomial time
via $({\rm V}_{=k})$.

Besides $({\rm V}_{=k})$,
the problem $({\rm V}_{c})$ has some similarity to 
other
problems discussed so far. 
It first looks similar to $({\rm V}^n(w))$ with $n=2$, 
but in fact they are different in the sense that 
the additional costs $c(|X_1 \cap X_2|)$ 
and $w(X_1 \cap X_2)$ cannot represent each other: 
they just coincide in the cases where $c$ is a linear function and 
$w(v)$ is identical for every $v \in V$. 

This observation, however, 
leads to the following generalization of $({\rm V}_{c})$.
That is, 
we can generalize 
$({\rm V}_{c})$ 
to have $n$ valuated matroids instead of the two weight functions $w_1$ and $w_2$, 
and can solve it in strongly polynomial time 
by our solution to $({\rm V}^n(w))$ 
if $c$ is linear and nonnegative.
Similarly,
we can generalize
$({\rm V}_{c})$ to have two M-convex functions on $\Z_+^V$, 
and can solve it in weakly polynomial time 
via $({\rm M}_{\geq k}(w))$ if $c$ is linear and nonpositive.

\subsection{Combinatorial optimization problem with interaction costs}\label{subsec:COPIC}

Lendl, \'Custi\'c, and Punnen~\cite{LCP19} introduced 
a framework of \emph{combinatorial optimization with interaction costs} (\emph{COPIC}),  
which is 
described as follows. 
For two sets $V_1$ and $V_2$, 
we are given cost functions $w_1\colon V_1 \to \R$ and 
$w_2 \colon V_2 \to \R$, 
as well as 
\emph{interaction costs} $q\colon V_1 \times V_2 \to \R$. 
The objective is to find a pair of feasible sets $X_1 \subseteq V_1$ and $X_2 \subseteq V_2$ 
minimizing 
$$
\sum_{u\in X_1}w_1(u) + \sum_{v \in X_2}w_2(v) + \sum_{u\in X_1}\sum_{v \in X_2}q(u,v).
$$

We focus on the 
\emph{diagonal COPIC}, 
where $V_1$ and $V_2$ are identical and 
$q(u,v)=0$ if $u \neq v$. 
We further assume that 
the feasible sets are the base families of matroids. 
That is,   
the problem is formulated by 
two matroids $(V, \B_1)$ and $(V,\B_2)$ 
and 
modular cost functions $w_1,w_2,q \colon 2^V \to \R$ 
in the following way:
\begin{align}\label{eq:copic}
&\begin{array}{lll}
\text{Minimize} \quad& w_1(X_1) + w_2(X_2) + q(X_1 \cap X_2)\\
\text{subject to} \quad& X_i \in \mathcal{B}_i \quad &(i = 1,2). 
\end{array} 
\end{align}

The problem~\eqref{eq:copic} appears in the context of the {\it $2$-min-max-min robustness}~\cite{MPA/BK17} defined as follows:
\begin{align}\label{eq:2-MMM}
    \displaystyle
    \min_{X_1, X_2\in \B} \max_{w \in \mathcal{W}} \min \{ w(X_1), w(X_2) \},
\end{align}
where $\B \subseteq 2^V$ is the base family of a matroid and $\mathcal{W}$ is a family of cost functions on $V$.
Chassein and Goerigk~\cite{DAM/CG20} showed that,
if $\mathcal{W}$ 
is
of the form
\begin{align}\label{eq:W}
    \mathcal{W} = \left\{
    w \colon  V \to \R_+
    \ \middle|\ \underline{w}(v) \leq w(v) \leq \overline{w}(v) \ (v \in V),\ \sum_{v \in V} \frac{w(v) - \underline{w}(v)}{ \overline{w}(v) - \underline{w}(v) } \leq C \right\}
\end{align}
for some $C \in \R_+$,
then
the $2$-min-max-min robustness
can be reduced to $O(|V|^3)$ many problems 
in
the form~\eqref{eq:copic}, 
where $\B_1 = \B_2$, $w_1$ and $w_2$ are nonpositive weight functions,
and $q$ is a nonnegative weight function on $V$.

Chassein and Goerigk~\cite{DAM/CG20}
proved that the above special case of the problem~\eqref{eq:copic}
can be solved in polynomial time via the ellipsoid method. 
Other previous work on the problem~\eqref{eq:copic}
includes the following.
If 
$w_1$ and $w_2$ are identically zero 
and 
$q \ge 0$, 
then the problem~\eqref{eq:copic} amounts 
to finding a socially optimal state in a two-player 
matroid congestion game, 
and thus 
can be solved in polynomial time~\cite{ARV08}. 
Lendl et al.~\cite{LCP19} extended the solvability 
to the case where the interaction cost $q$ may be arbitrary.

Now we can discuss another direction of generalization: 
the costs $w_1$ and $w_2$ are valuated matroids. 
This is a special case of $({\rm V}^n(w))$
and $({\rm M}_{\geq k}(w))$, 
and thus can be solved in polynomial time when $q \ge 0$ or $q \leq 0$. 

\begin{theorem}
The problem~\eqref{eq:copic} 
can be solved in strongly polynomial time 
if 
$w_1,w_2 \colon 2^V \to \R$ are valuated matroids, 
and $q \ge 0$ or $q \le 0$. 
\end{theorem}

As mentioned above, Chassein--Goerigk's algorithm for the problem~\eqref{eq:copic} corresponding to the 2-min-max-min robustness
is based on the ellipsoid method
and hence not combinatorial. 
In contrast,
our algorithm 
provides a combinatorial algorithm for this case of the problem \eqref{eq:copic}. 
More generally,
our result leads to
the first combinatorial algorithm for the 
2-min-max-min robustness~\eqref{eq:2-MMM} 
with $\mathcal{W}$ 
in
the form~\eqref{eq:W}.

\subsection{Socially optimal states in valuated matroid congestion games}
\label{sec:congestion}

We finally present an application of $({\rm V}^n(w))$ in \emph{congestion games} \cite{Ros73}, 
a class of noncooperative games in game theory. 
A congestion game  is represented by a tuple $(N,V,(\B_i)_{i \in N}, (c_v)_{v \in V} )$, 
where 
$N = \{1,2, \ldots, n\}$ is a set of players, 
$V$ is a set of resources, 
$\B_i \subseteq 2^V$ is the set of strategies of a player $i \in N$, 
and 
$c_v\colon \Z_+ \to \R_+$ is a nondecreasing cost function 
associated with a resource $v \in V$.
A \emph{state} $\X=(X_1,X_2,\ldots, X_n)$ is a collection of 
strategies of all players, 
i.e.,\  
$X_i \in \B_i$ for each $i \in N$. 
For a state $\X=(X_1,X_2,\ldots, X_n)$, 
let $x^{(v)}(\X)$ denote the number of players using $v$, 
i.e.,\ 
$x^{(v)}(\X)=|\{ i \in N \mid v \in X_i \}|$. 
If $\X$ is clear from the context, 
$x^{(v)}(\X)$ is abbreviated as $x^{(v)}$. 
In a state $\X$, 
every player using a resource $v \in V$ should pay 
$c_v(x^{(v)})$ to use $v$, 
and thus the total cost paid by a player $i \in N$ is $\sum_{v \in X_i}c_v(x^{(v)})$. 
In a \emph{player specific-cost} model, the cost paid by a player $i \in N$ for using $v\in V$ is represented by a function $c_{i,v}\colon \Z_+ \to \R_+$, which may vary with each player.

The importance of congestion games is appreciated through the fact 
that the class of congestion games coincides with that of \emph{potential games}. 
Rosenthal~\cite{Ros73} proved that every congestion game is a potential game, 
and conversely, 
Monderer and Shapley~\cite{MS96} proved that 
every potential game is represented by a congestion game with the same potential function.

Here we show that, 
in a certain generalized model of 
matroid congestion games
with player-specific costs, 
computing 
a socially optimal state 
reduces to (a variant of) $({\rm V}^n(w))$.  
A state $\X^*=(X_1^*,X_2^*,\ldots, X_n^*)$ is called \emph{socially optimal} if 
the sum of the costs paid by all the players is minimum, 
i.e.,\ 
\begin{align*}
    \sum_{i \in N}\sum_{v \in X^*_i}c_v(x^{(v)}(\X^*)) \le \sum_{i \in N}\sum_{v \in X_i}c_v(x^{(v)}(\X))
\end{align*}
for any state $\X=(X_1,X_2,\ldots,X_n)$. 
In a \emph{matroid congestion game},  
the set $\B_i \subseteq 2^V$ of the strategies of each player $i\in N$ is 
the base family of a matroid on $V$. 
A socially optimal state in matroid congestion games can be computed in polynomial time 
if the cost functions are \emph{weakly convex}~\cite{ARV08,WS00}, 
while 
it is NP-hard for general nondecreasing cost functions~\cite{ARV08}. 
A function $c \colon \Z_+ \to \R$ is called \emph{weakly convex} if 
$(x+1)\cdot c(x+1) - x\cdot c(x)$ is nondecreasing for each $x \in \Z_+$. 

We consider the following generalized model of congestion games with player-specific costs. 
In a state $\X=(X_1,X_2, \ldots, X_n)$, 
the cost paid by a player $i \in N$ is 
\begin{align}
\label{EQcost}
    \omega_i(X_i) + \sum_{v \in X_i}d_v (x^{(v)}), 
\end{align}
where 
$\omega_i \colon 2^V \to \R_+$ is a monotone set function 
and 
$d_v\colon \Z_+ \to \R_+$ is a nondecreasing function for each $v \in V$. 
This model represents a situation where 
a player $i \in N$ should pay $\omega_i(X_i)$ regardless of the strategies of the other players, 
as well as 
$d_v(x^{(v)})$ for every resource $v \in X_i$, 
which is an additional cost resulting from the congestion on $v$. 
It is clear that 
the standard model of congestion games is a special case 
where 
$\omega_i(X_i)=\sum_{v \in X_i}c_v(1)$ for every $i \in N$ 
and every $X_i \in \B_i$, 
and 
$$d_v(x) = 
\begin{cases}
0 & (x=0), \\
c_v(x) - c_v(1) &(x\ge 1).
\end{cases}
$$

In this model, 
the sum of the costs paid by all the players is equal to
\begin{align}
\label{EQsociallyopt}
\sum_{i \in N}\omega_i(X_i) + \sum_{v\in V}x^{(v)} \cdot d_v(x^{(v)}).
\end{align}
The following lemma is straightforward to see. 
\begin{lemma}\label{lem:convex}
The following are equivalent. 
\begin{itemize}
    \item $c_v$ is weakly convex. 
    \item $d_v$ is weakly convex. 
    \item $x \cdot d_v$ is discrete convex. 
\end{itemize}
\end{lemma}
It follows from Lemma~\ref{lem:convex} that,
if $c_v$ (or $d_v$) is weakly convex,
then
the function $\sum_{v\in V}x_v \cdot d_v(x^{(v)})$
is laminar convex.

The solution for $({\rm V}^n(w))$,
or 
the DCA perspective for~$({\rm V}^n(w))$,
provides a new insight on this model of cost functions in matroid congestion games. 
In addition to the weak convexity of $d_v$ $(v \in V)$, 
this model allows us to 
introduce some convexity of the cost function $\omega_i$. 
Namely, 
we can assume that $\omega_i$ is a valuated matroid for every $i \in N$. 
Then, 
computing the optimal state, 
i.e., 
minimizing~\eqref{EQsociallyopt}, is naturally viewed as
valuated matroid intersection problem
for the valuated matroid $\sum_{i \in N} \omega_i(X_i)$ and
the laminar convex function $\sum_{v\in V}x_v \cdot d_v(x_v)$
as in $({\rm V}^n(w))$.
Thus it can be done in polynomial time. 

\begin{theorem}
In a matroid congestion game in which 
each player's cost is represented by 
\eqref{EQcost}, 
the socially optimal state can be 
computed in strongly polynomial time 
if 
$\omega_i$ is a valuated matroid 
for each player $i \in N$ 
and 
$d_v$ is weakly convex for each resource $v \in V$. 
\end{theorem}

\section{Discussions}\label{sec:discussion}
In this paper, 
we have 
presented
several types of 
minimization of the sum of valuated matroids (or M-convex functions) under intersection constraints.
We here consider 
the following another natural generalization of $({\rm V}_{=k})$, 
where 
$\omega_1$ and $\omega_2$ are valuated matroids on $2^V$,
$w$ is a weight function on $V$, 
and $k$ is a nonnegative integer: 
\begin{align*}
({\rm V}_{=k}(w)) \qquad
\begin{array}{ll}
\text{Minimize} \quad& \omega_1(X_1) + \omega_2(X_2) + w(X_1 \cap X_2)\\
\text{subject to} \quad& |X_1 \cap X_2| = k. 
\end{array}
\end{align*}

The problem $({\rm V}_{=k}(w))$
is similar to $\VIAP(k)$, 
but 
is essentially different. 
A problem 
that is similar to $({\rm V}_{=k}(w))$ and 
can be formulated by 
$\VIAP(k)$ is the following: 
\begin{align}
\label{eq:21}
\begin{array}{ll}
\text{Minimize} \quad& \omega_1(X_1) + \omega_2(X_2) + w(F)\\
\text{subject to} \quad& F \subseteq X_1 \cap X_2,\\
&|F| = k. 
\end{array}
\end{align}
The difference 
between 
the problems $({\rm V}_{=k}(w))$ and \eqref{eq:21}
is that 
$|X_1 \cap X_2|$
should be exactly equal to $k$ and all elements in $X_1 \cap X_2$ affect the objective value in $({\rm V}_{=k}(w))$, 
whereas 
$|X_1 \cap X_2|$ is just required to be at least $k$ and only $k$ elements in
$X_1 \cap X_2$ affect the objective value in \eqref{eq:21}.

While 
$\VIAP(k)$ (and hence \eqref{eq:21}) can be solved in polynomial time, 
the complexity of $({\rm V}_{=k}(w))$ is open 
even when the cardinality constraint $|X_1 \cap X_2| = k$ is removed
and 
$\omega_1$ and $\omega_2$
are modular functions on the base families of some matroids.
For $({\rm V}_{=k}(w))$, 
only the following cases are known to be tractable:
\begin{itemize}
\item 
If 
$w$ is identically zero, then 
$({\rm V}_{=k}(w))$ is equivalent to $({\rm V}_{=k})$.
\item 
If $w \geq 0$
and
the cardinality constraint $|X_1 \cap X_2| = k$ is removed, 
then 
$({\rm V}_{=k}(w))$ is a subclass of $({\rm V}^n(w))$ with $w \ge 0$.
\item
If $w \leq 0$
and
$|X_1 \cap X_2| = k$ is replaced by $|X_1 \cap X_2| \geq k$,
then $({\rm V}_{=k}(w))$
is a subclass of $({\rm M}_{\geq k}(w))$ 
with $w \le 0$
\item 
If 
$|X_1 \cap X_2| = k$ is removed
and 
$\omega_1$ and $\omega_2$
are the indicator functions of the base families of some matroids, 
then 
$({\rm V}_{=k}(w))$
has been dealt with Lendl et al.~\cite{LCP19};
see~Section~\ref{subsec:COPIC}.
\end{itemize}

Another possible direction of research would be to generalize our framework so that 
it includes computing the socially optimal state of 
\emph{polymatroid congestion games}~\cite{HKP18,Tak19}, 
as 
we have done for matroid congestion games in~Section~\ref{sec:congestion}. 
Polymatroid congestion games offer a model generalizing matroid congestion games where the usage of a resource by a player may not be binary, 
and its multiplicity can be represented by a nonnegative integer. In this model, the sum of the costs paid by all players for a resource $v$ may no longer be represented as $x^{(v)} \cdot d_v(x^{(v)})$ as in \eqref{EQsociallyopt}, 
because the number of players using $v$ may not be equal to the multiplicity of the usage of $v$.

\section*{Acknowledgements}
We thank Andr\'{a}s Frank and
Kazuo Murota
for careful reading and numerous helpful comments
and
Marc Goerigk
for bibliographical information on~\cite{DAM/CG20}.
We also thank
the referees for helpful comments.
The first author was supported by JSPS KAKENHI Grant Numbers JP17K00029, JP19J01302, 20K23323, 20H05795, Japan.
The second author was supported by JSPS KAKENHI Grant Numbers JP16K16012, JP26280004,
JP20K11699, Japan.
This is a post-peer-review, pre-copyedit version of an article published in Mathematical Programming. The final authenticated version is available online at:  \url{https://doi.org/10.1007/s10107-021-01625-2}.

\appendix

\section{Algorithm description and complexity analyses}\label{appendix:algorithm description}
In this section,
we describe the full behavior of the algorithm for $({\rm V}_{\ge k})$. 
We then analyse the time complexities of our algorithms 
for $({\rm V}_{\ge k})$ and $({\rm V}_{= k})$ to prove  Theorem \ref{thm:V=k}. 

Suppose that,
for some nonnegative integer $i < k$,
we have at hand an optimal solution $(X_1, X_2)$ for~$({\rm V}_{\geq i})$ and its optimality witness $(p_1, p_2, F)$;
recall that the optimality condition has been described in Lemma~\ref{lem:witness}.
We then find an optimal solution for 
$({\rm V}_{\geq i+1})$ and its optimality witness 
by utilizing an auxiliary digraph
$\vec{G}=(V_1 \cup V_2 \cup \{s, t\}, A)$
endowed with a nonnegative arc-length function $\ell$ on $A$
defined 
from $(X_1, X_2)$ and $(p_1, p_2, F)$ 
as follows.
Here $s$ and $t$ are new vertices, 
which play 
roles as a source vertex and sink vertex,
respectively.
The arc set $A$ is defined by
\begin{align*}
    A := \vec{E} \cup \vec{F} \cup A_1 \cup A_2 \cup S \cup T,
\end{align*}
where
\begin{align*}
    \vec{E} &:= \{ (v^1, v^2) \mid v \in V \},\\
    \vec{F} &:= \{ (v^2, v^1) \mid v \in F \},\\
    A_1 &:= \{ (u^1, v^1) \mid u^1 \in X_1,\ v^1 \in V_1 \setminus X_1,\ X_1 \setminus \{u^1\} \cup \{v^1\} \in \dom \omega_1 \},\\
    A_2 &:= \{ (v^2, u^2) \mid u^2 \in X_2,\ v^2 \in V_2 \setminus X_2,\ X_2 \setminus \{u^2\} \cup \{v^2\} \in \dom \omega_2 \},\\
    S &:= \{ (s, v^1) \mid v^1 \in X_1 \setminus X_2 \},\\
    T &:= \{ (v^2, t) \mid v^2 \in X_2 \setminus X_1 \}.
\end{align*}
Define the arc-length function 
$\ell:A\to \R$
by
\begin{align*}
    \ell(a) :=
    \begin{cases}
    (\omega_1 - p_1)(X_1 \setminus \{u^1\} \cup \{v^1\}) - (\omega_1 - p_1)(X_1) & \text{if $a = (u^1, v^1) \in A_1$},\\
    (\omega_2 + p_2)(X_2 \setminus \{u^2\} \cup \{v^2\}) - (\omega_2 + p_2)(X_2) & \text{if $a = (v^2, u^2) \in A_2$},\\
    0 & \text{if $a \in A \setminus (A_1 \cup A_2)$}.
    \end{cases}
\end{align*}
We remark that $\ell$ is nonnegative since $X_1$ and $X_2$ are minimizers of $\omega_1 - p_1$ and $\omega_2 + p_2$, respectively.

The augmenting path algorithm for~$({\rm V}_{\geq k})$ runs in the auxiliary digraph as follows:
\begin{description}
	\item[Step 1:]
	Let $X_1$ and $X_2$ be the minimizers of $\omega_1$ and $\omega_2$,
respectively.
\begin{itemize}
    \item
    If $|X_1 \cap X_2| \geq k$,
	then output $(X_1,X_2)$ and stop.
	\item
	Otherwise, let $k' := |X_1 \cap X_2| < k$
	and let $p_1, p_2$ be potential functions defined by $p_1(v^1) = p_2(v^2) = 0$ for all $v \in V$.
	Then $(X_1, X_2)$ is an optimal solution for
	$({\rm V}_{\geq k'})$
	with optimality witness $(p_1, p_2, X_1 \cap X_2)$.
\end{itemize}
	\item[Step 2:]
	While 
	$i=|X_1 \cap X_2| < k$,
	do the following.

	Suppose here that
	$(X_1, X_2)$ is an optimal solution for $({\rm V}_{\geq i})$
	with optimality witness $(p_1, p_2, X_1 \cap X_2)$.
	Let $\vec{G}$ be the auxiliary digraph for $(X_1, X_2)$ with $(p_1, p_2, X_1 \cap X_2)$.
	\begin{itemize}
	    \item
	    If there is no $s$-$t$ path in $\vec{G}$,
	    then output ``$({\rm V}_{\geq k})$ is infeasible'' and stop.
	    \item
	    Otherwise, find a shortest $s$-$t$ path $P$ in $\vec{G}$ with respect to the arc length $\ell$ such that the number of edges is smallest among the shortest $s$-$t$ paths.
	    For $x \in V_1 \cup V_2 \cup \{t\}$,
	    let $d(x)$ be the length of the shortest $s$-$x$ path 
	    with respect to $\ell$ 
	    in $\vec{G}$,
	    where $d(x) := +\infty$ if there is no $s$-$x$ path.
	    Update $X_1$, $X_2$, $p_1$, and $p_2$ by
	    \begin{align*}
	        X_1 \ &\leftarrow\ X_1 \setminus (P \cap X_1) \cup (P \cap (V_1 \setminus X_1)),\\
	        X_2 \ &\leftarrow\ X_2 \setminus (P \cap X_2) \cup (P \cap (V_2 \setminus X_2)),\\
	        p_1(v^1) \ &\leftarrow\ p_1(v^1) + \min \{d(v^1), d(t) \} \quad \text{for $v^1 \in V_1$},\\
	        p_2(v^2) \ &\leftarrow\ p_2(v^2) - \min \{ d(v^2), d(t) \} \quad \text{for $v^2 \in V_2$}.
	    \end{align*}
	    Then, 
        the resulting $X_1$, $X_2$, $p_1$, and $p_2$ satisfy that $|X_1 \cap X_2| = i + 1$
        and 
	    $(X_1,X_2)$ is an optimal solution for $({\rm V}_{\geq i+1})$
	    with optimality witness $(p_1, p_2, X_1 \cap X_2)$.
	\end{itemize}
	\item[Step 3:]
	Output $(X_1, X_2)$ and stop.
\end{description}

We are ready to prove Theorem~\ref{thm:V=k}.
We first see the time complexity of the algorithm for $({\rm V}_{\geq k})$.
We can obtain minimizers of $\omega_1$ and $\omega_2$ in $O(|V| r \gamma)$ time~\cite{AML/DW90},
and hence 
Step~1 can be done in $O(|V| r \gamma)$ time.
Consider each iteration in Step~2.
We can construct the auxiliary graph $\vec{G}$ in $O(|V|r\gamma)$ time.
Indeed,
the number of edges in $\vec{G}$ is $O(|V|r)$: 
$|\vec{E}|$ is $O(|V|)$, $|\vec{F}|$, $|S|$, and $|T|$ are $O(|r|)$,
$|A_1|$ is $O(r_1(|V| - r_1)) = O(|V|r)$, 
and
$|A_2|$ is $O(r_2(|V| - r_2)) = O(|V|r)$,
where $r_1$ and $r_2$ are the ranks of $\omega_1$ and $\omega_2$,
respectively.
Furthermore the arc-length function $\ell$ can be computed in $O((|A_1| + |A_2|) \gamma) = O(|V|r \gamma)$ time.
Then we can compute the length $d(x)$ of the shortest $s$-$x$ path for each $x \in V_1 \cup V_2 \cup \{s\}$
in $O(|V|r + |V| \log |V|)$ time
by using Dijkstra's algorithm with Fibonacci heaps~\cite{JACM/FT87} (see also~\cite[Section~7.4]{book/Schrijver03}).
The update of $(X_1, X_2)$ and $(p_1, p_2)$ 
requires
$O(|V|)$ time.
Therefore
each iteration in Step~2 takes $O(|V|r\gamma + |V| \log |V|)$ time.
Since the number of iterations is at most $k$,
Step~2 can be done in $O(|V|r k \gamma + |V| k \log |V|)$ time.
Thus the running-time of the algorithm for $({\rm V}_{\geq k})$ is $O(|V|r k \gamma + |V| k \log |V|)$.

We then consider the time complexity of the algorithm for $({\rm V}_{= k})$.
Let $X_1$ and $X_2$ be minimizers of $\omega_1$ and $\omega_2$,
respectively,
which can be obtained in $O(|V| r \gamma)$ time.
If $|X_1 \cap X_2| \leq k$,
then the algorithm for $({\rm V}_{= k})$ is the same as that for $({\rm V}_{\geq k})$,
and hence
the running-time is $O(|V|r k \gamma + |V| k \log |V|)$. 
If $|X_1 \cap X_2| > k$,
we solve$({\rm V}_{\geq r_1 - k})$ for $\omega_1$ and $\overline{\omega_2}$.
By the same argument as for the algorithm for $({\rm V}_{\geq k})$,
one can see that the running-time of $({\rm V}_{\geq r_1 - k})$ is $O(|V|r (r_1-k) \gamma + |V| (r_1-k) \log |V|)$.
Altogether,
$({\rm V}_{= k})$ can be solved in $O(|V|r^2 \gamma + |V| r \log |V|)$ time.

\section{Anothor solution to $({\rm V}_{= k})$}\label{appendix:Laszlo}
This section provides another solution to $({\rm V}_{= k})$,
which is mentioned at the end of Section~\ref{subsec:strongly poly-time}.
We first solve $({\rm V}_{\leq k})$ and $({\rm V}_{\geq k})$.
The former is a special case of $({\rm V}_{\I}^n)$ in which $n = 2$ and $\I$ is the independent set family of the uniform matroid with rank $k$,
and can be solved in $O(|V|r^2 \gamma + |V| r \log |V|)$ time by Theorem~\ref{thm:VI}.
The latter can also be solved in $O(|V|r^2 \gamma + |V| r \log |V|)$ time by 
Theorem~\ref{thm:V=k}
(or the reduction to valuated matroid intersection).
If the output optimal solution $(X_1, X_2)$ of $({\rm V}_{\leq k})$ satisfies $|X_1 \cap X_2| = k$
or that of $(X_1', X_2')$ of $({\rm V}_{\geq k})$ satisfies $|X_1' \cap X_2'| = k$,
then we are done.

Otherwise, 
we have that $X_1$ and $X_1'$ are minimizers of $\omega_1$, and $X_2$ and $X_2'$ are those of $\omega_2$.
Indeed, suppose, to the contrary, that $X_1$ is not a minimizer of $\omega_1$.
Then there are $u \in X_1$ and $v \in V \setminus X_1$ such that $\omega_1(X_1 \setminus \{u\} \cup \{v\}) < \omega_1(X_1)$.
By $|X_1 \cap X_2| < k$,
it follows that
$|(X_1 \setminus \{u\} \cup \{v\}) \cap X_2| \leq k$.
Therefore, 
$(X_1 \setminus \{u\} \cup \{v\}, X_2)$ is feasible and $\omega_1(X_1 \setminus \{u\} \cup \{v\}) + \omega_2(X_2) < \omega_1(X_1) + \omega_2(X_2)$, 
which
contradicts the optimality of $(X_1, X_2)$.
By the same argument, 
we can conclude that
$X_1'$ is also a minimizer of $\omega_1$
and $X_2, X_2'$ are minimizers of $\omega_2$.

Let $p = |X_1' \setminus X_1|$ and 
$q= |X_2' \setminus X_2|$.
Since the set of minimizers of a valuated matroid
forms a base family of some matroid,
there is a sequence 
$(X_1^0, X_1^1, \dots, X_1^p)$
of minimizers of $\omega_1$
such that
$X_1^0 = X_1$ and
$|X_1' \setminus X_1^{i}| = p-i$ for
each $i = 0,1,\dots, p$.
Note 
that
$|X_1^{i-1} \cap X_2| - 1 \leq |X_1^i \cap X_2| \leq |X_1^{i-1} \cap X_2| + 1$.
If 
$|X_1^i \cap X_2| = k$ holds for some $i \in [p]$,
then output 
$(X_1^i, X_2)$.
Otherwise, we have that
$|X_1' \cap X_2| < k$. 
Then we can
similarly consider 
a sequence $(X_2^0, X_2^1, \dots, X_2^q)$
of minimizers of $\omega_2$, 
where $X_2^0 = X_2$
and $|X_2' \setminus X_2^{j}| = q-j$ for
each $j = 0,1,\dots,q$.
It then follows from $|X_1' \cap X_2| < k$ and $|X_1' \cap X_2'| > k$ that
there is $X_2^j$
with $|X_1' \cap X_2^j| = k$.
We thus output 
$(X_1' ,X_2^j)$
as an optimal solution.

\section{The algorithm for $({\rm W}_{=k})$ by Lendl et al.}\label{appendix:LPT}
The primal-dual algorithm of Lendl et al.~\cite{arxiv/LPT19} is described as follows.
Let $X_1 \in \argmin_{X \in \B_1} w_1(X)$ and $X_2 \in \argmin_{X \in \B_2} w_2(X)$,
respectively. 
\begin{description}
    \item[Case 1 ($|X_1 \cap X_2| \leq k$):]
    Let $k' := |X_1 \cap X_2|$
    and
    let $q_1, q_2$ be 
    potential functions
    defined by $q_1(v^1) = q_2(v^2) = 0$ for all $v \in V$.
    
    While $|X_1 \cap X_2| < k$,
    do the following.
    Let $\vec{G}$ be the auxiliary graph for $(X_1, X_2)$ with its LPT optimality witness $(q_1, q_2)$.
    \begin{itemize}
        \item
        Suppose that there is a zero length $s$-$t$ path in $\vec{G}$ with respective to $\ell$.
        Then 
        let $P$ be a zero length $s$-$t$ path 
        with the smallest number of edges, and update 
        $X_1$ and $X_2$ by
        \begin{align*}
	        X_1 &\leftarrow X_1 \setminus (P \cap X_1) \cup (P \cap (V_1 \setminus X_1)),\\
	        X_2 &\leftarrow X_2 \setminus (P \cap X_2) \cup (P \cap (V_2 \setminus X_2)).
	    \end{align*}
        \item
        Suppose that there is no zero length $s$-$t$ path in $\vec{G}$.
        Let $R$ be the set of vertices reachable from $s$ with a zero length path,
        and $\delta := \min \{ \ell(a) \mid a \in A_1 \cup A_2,\ \ell(a) > 0 \}$.
        
        If $\delta = +\infty$,
        then output ``$({\rm W}_{=k})$ is infeasible.''
        Otherwise
        update
        $q_1$ and $q_2$ by
        \begin{align*}
	        q_1(v^1) &\leftarrow
	        \begin{cases}
	        q_1(v^1) + \delta & \text{if $v^1 \in V_1 \cap R$},\\
	        q_1(v^1) & \text{if $v^1 \in V_1 \setminus R$},
	        \end{cases}\\
	        q_2(v^2) &\leftarrow
	        \begin{cases}
	        q_2(v^2) & \text{if $v^2 \in V_2 \cap R$},\\
	        q_2(v^2) - \delta & \text{if $v^2 \in V_2 \setminus R$}.
	        \end{cases}
	    \end{align*}
    \end{itemize}
	
	\item[Case 2 ($|X_1 \cap X_2| > k$):]
	Let $r_1$ be the rank of $M_1 = (V, \B_1)$, and let $\overline{w_2} := - w_2$ and $\overline{M_2} := (V, \{ V \setminus B \mid B \in \B_2 \})$.
	Then
	apply Case 1 to $({\rm W}_{= r_1 - k})$ with $(w_1, \overline{w}_2; M_1, \overline{M_2})$. 
\end{description}

We discuss the difference 
of this algorithm and our algorithm 
in the
updating procedures.
The algorithm of Lendl et al.\ 
only considers the zero length edges in the 
auxiliary digraph $\vec{G}$
in updating a solution or 
potential functions,
while our algorithm considers all edges in $\vec{G}$ 
to find the shortest paths with respect to $\ell$.
Furthermore, 
as mentioned in Section \ref{SEC:LPT}, 
in the algorithm of Lendl et al.,\ 
the update phases of a solution
and of 
potential functions
are completely separated, 
while our algorithm 
simultaneously
updates a solution and 
potential functions. 
That is,
for
one update of a solution,
the algorithm of Lendl et al.\ requires
at most $|V|$ updates of 
potential functions, 
while our algorithm requires only one update of 
potential functions.

The difference 
in
the running-times of the algorithms
follows from the above arguments.
Here we remark that, in $({\rm W}_{\geq k})$,
we can compute the edge length $\ell(a)$ of $a$ in constant time for each $a$.
In 
the algorithm of Lendl et al.,
one update of a solution takes $O(|V|^2 r)$ time,
since each 
update of potential functions
requires $O(|V|r)$ time.
Hence the time complexity of 
their
algorithm
is $O(|V|^2 r^2)$.
On the other hand,
our algorithm requires $O(|V|r + |V| \log |V|)$ time 
for one update of a solution.
Thus, 
the running-time of our algorithm is 
$O(|V|r^2 + |V| r \log |V|)$ time by 
Theorem~\ref{thm:V=k} with $\gamma = O(1)$.

\end{document}